\documentclass[11pt]{amsart}
\usepackage{times}      % use Times fonts for text
\usepackage{latexsym}   % use all LaTeX fonts
\usepackage{amssymb}    % use all AMS fonts
\usepackage{amsmath}    % include some AMS-LaTeX functionality
\usepackage{amsthm}     % AMS style theorem and proof environments
\usepackage{mathrsfs}   % special calligraphic characters like \mathscr{S}
\begin{document}
% begin top matter
% ********************* macroes needed for this paper ************************
\newtheorem{Pro}{Proposition}[section]
\newtheorem{Lem}[Pro]{Lemma}
\newtheorem{Sub}[Pro]{Sublemma}
\newtheorem{Thm}[Pro]{Theorem}
\newtheorem{MThm}{Theorem}
\renewcommand{\theMThm}{\Alph{MThm}}
\newtheorem{Rem}[Pro]{Remark}
\newtheorem{Def}[Pro]{Definition}
\newtheorem{Not}{Notation}
\newtheorem{Cor}[Pro]{Corollary}
\newtheorem{Ithm}{Theorem}
\renewcommand{\theIthm}{\arabic{chapter}.\arabic{Ithm}}
\newtheorem{IndThm}{Theorem}
\newtheorem{Idef}{Definition}
\newcommand{\SL}{\mathcal L^{1,p}(\Om)}
\newcommand{\Lp}{L^p(\Omega)}
\newcommand{\CO}{C^\infty_0(\Omega)}
\newcommand{\Rn}{\mathbb R^n}
\newcommand{\Rm}{\mathbb R^m}
\newcommand{\R}{\mathbb R}
\newcommand{\Om}{\Omega}
\newcommand{\Hn}{\mathbb H^n}
\newcommand{\HH}{\mathbb H^1}
\newcommand{\eps}{\epsilon}
\newcommand{\BVX}{BV_H(\Omega)}
\newcommand{\IO}{\int_\Omega}
\newcommand{\bG}{\boldsymbol{G}}
\newcommand{\bg}{\mathfrak g}
\newcommand{\p}{\partial}
\newcommand{\Xnu}{\overset{\rightarrow}{ X_\nu}}
\newcommand{\nuX}{\boldsymbol{\nu}^H}
\newcommand{\Up}{\boldsymbol{\mathcal Y}_H}
\newcommand{\n}{\boldsymbol \nu}
\newcommand{\sigmau}{\boldsymbol{\sigma}^u_H}
\newcommand{\di}{\nabla^{H,\mS}_i}
\newcommand{\duno}{\nabla^{H,\mS}_1}
\newcommand{\ddue}{\nabla^{H,\mS}_2}
\newcommand{\del}{\nabla^{H,S}}
\newcommand{\nui}{\nu_{H,i}}
\newcommand{\nuj}{\nu_{H,j}}
\newcommand{\dej}{\delta_{H,j}}
\newcommand{\cx}{\boldsymbol{c}^\mathcal S}
\newcommand{\sx}{\sigma_H}
\newcommand{\lx}{\mathcal L_H}
\newcommand{\pb}{\overline p}
\newcommand{\qb}{\overline q}
\newcommand{\ob}{\overline \omega}
\newcommand{\nuu}{\boldsymbol \nu_{H,u}}
\newcommand{\nuv}{\boldsymbol \nu_{H,v}}
\newcommand{\Bl}{\Bigl|_{\lambda = 0}}
\newcommand{\mS}{\mathcal S}
\newcommand{\delh}{\Delta^H}
\newcommand{\delinf}{\Delta_{H,\infty}}
\newcommand{\nabh}{\nabla^H}
\newcommand{\delp}{\Delta_{H,p}}
\newcommand{\mO}{\mathcal O}
\newcommand{\delhs}{\Delta_{H,S}}
\newcommand{\lhs}{\hat{\Delta}_{H,S}}
\newcommand{\bN}{\boldsymbol{N}}
\newcommand{\bnu}{\boldsymbol \nu}
\newcommand{\la}{\lambda}
\newcommand{\nup}{({\boldsymbol{\nu}^H})^\perp}
\newcommand{\oX}{\overline X}
\newcommand{\oY}{\overline Y}
\newcommand{\ou}{\overline u}
\newcommand{\dels}{\nabla^{H,\mS}}
\newcommand{\delXY}{\nabla^{H,\mS}_X Y}
\newcommand{\delYX}{\nabla^{H,\mS}_Y X}
\newcommand{\hp}{\tilde{h}_0'(s)}
\newcommand{\uc}{u_{,ij}}
\newcommand{\sij}{\sum_{i,j=1}^m}
\newcommand{\sul}{\Delta^H}
\newcommand{\nor}{\boldsymbol \nu}
\newcommand{\fv}{\mathcal V^{H}_I(S;\mathcal X)}
\newcommand{\sv}{\mathcal V^{H}_{II}(S;\mathcal X)}

\newcommand{\Id}{I_\delta}
\newcommand{\Ri}{\mathbb{R} \times I}
\newcommand{\Rd}{\mathbb{R} \times \Id}

\newcommand{\til}{\tilde{\boldsymbol{\nu}}^H}
\newcommand{\tilp}{(\tilde{\boldsymbol{\nu}}^H)^\perp}
\newcommand{\tilL}{\tilde{\mathcal{L}}_z(r)}
\newcommand{\tilLg}{\tilde{\mathcal{L}}_{\gamma(s)}(r)}
\newcommand{\tilLp}{\tilde{\mathcal{L}}_z'(r)}
\newcommand{\tilLzero}{\tilde{\mathcal{L}}_z(0)}
\newcommand{\ra}{\rightarrow}
\renewcommand{\to}{\rightarrow}
\newcommand{\ip}{\langle \cdot,\cdot\rangle}
\newcommand{\minus}{\setminus}
\newcommand{\Fx}{F_{\mathcal X}}
% ****************************************************************************

\newcommand{\den}{(1+( W^\phi  \phi)^2)^\frac{3}{2}}

\title[Stable complete embedded minimal surfaces in $\HH$, etc.]{Stable complete embedded minimal surfaces in $\HH$ with empty characteristic locus are vertical planes}
\author{D. Danielli}
\address{Department of Mathematics\\Purdue University \\
West Lafayette, IN 47907} \email[Donatella
Danielli]{danielli@math.purdue.edu}
\thanks{First author supported in part by NSF CAREER Grant, DMS-0239771}

\author{N. Garofalo}
\address{Department of Mathematics\\Purdue University \\
West Lafayette, IN 47907} \email[Nicola
Garofalo]{garofalo@math.purdue.edu}
\thanks{Second author supported in part by NSF Grant DMS-0701001}

\author{D. M. Nhieu}
\address{Department of Mathematics\\San Diego Christian College \\
San Diego, CA 92019} \email[Duy-Minh
Nhieu]{dnhieu@sdcc.edu}

\author{S. D. Pauls}
\address{Department of Mathematics, Dartmouth College, Hanover, NH 03755}
\email{scott.pauls@dartmouth.edu}
\thanks{Fourth author supported in part by NSF Grant DMS-0306752}

%\thanks{\today}

\begin{abstract}
In the recent paper \cite{DGNP} we have proved that the only stable
$C^2$ minimal surfaces in the first Heisenberg group $\Hn$ which are
graphs over some plane and have empty characteristic locus must be
vertical planes. This result represents a sub-Riemannian version of
the celebrated theorem of Bernstein.

\noindent In this paper we extend the result in \cite{DGNP} to $C^2$
complete embedded
minimal surfaces in $\HH$ with empty characteristic locus. We prove that every such a surface without boundary must be a vertical plane. \\

%\vspace{.2in}

%{\Large \bf ****DRAFT: NOT FOR GENERAL DISTRIBUTION****}
\end{abstract}

\maketitle
\section{\textbf{Introduction}}\label{S:intro}

The study of minimal surfaces has been one of the prime drivers of
the study of geometry and calculus of variations in the twentieth
century and, in particular, the Bernstein problem has played a
central role.  Bernstein proved his Theorem \cite{Bernstein}, that a
$C^2$ minimal graph in $\R^3$ must necessarily be an affine plane in
1915 and, almost fifty years later, a new insight of Fleming
\cite{Fle} generated renewed interest in the problem.  The work of
De Giorgi, \cite{DG3}, Almgren, \cite{Al}, Simons, \cite{Sim}, and
Bomberi-De Giorgi-Giusti, \cite{BDG}, culminated in the complete
solution to the Bernstein problem:

\begin{Thm}\label{T:classicalB}
Let $S = \{(x,u(x)) \in \R^{n+1}| x\in \Rn ,  x_{n+1} = u(x)\}$ be
a $C^2$ minimal graph in $\R^{n+1}$, i.e., let $u\in C^2(\Rn)$ be a
solution of the minimal surface equation
\begin{equation}\label{ms}
div\left(\frac{Du}{\sqrt{1 + |Du|^2}}\right) = 0,
\end{equation}
in the whole space. If $n\leq 7$, then there exist $a\in \Rn$,
$\beta \in \R$ such that $u(x) = <a,x> + \beta$, i.e., $S$ must
be an affine hyperplane. If instead $n\geq 8$, then there exist
non affine (real analytic) functions on $\Rn$ which solve
\eqref{ms}.
\end{Thm}

Roughly a decade later,  Fischer-Colbrie and Schoen, \cite{S-FC}, and do
Carmo and Peng, \cite{DoC-P}, imposing a stability condition,
independently proved
a far reaching generalization of the Bernstein property:
\begin{Thm}\label{T:classicalS}
Every stable complete minimal surface $S\subset \R^3$ must be a plane.
\end{Thm}

Here, stable means that on every compact set $S$
minimizes area up to order two. We note in passing that, thanks to
the strict convexity of the area functional $\mathcal A(u) =
\int_\Om \sqrt{1+|Du|^2} dx$, where $\Om\subset\subset \Rn$, for Euclidean graphs on $\R^n$ the
stability assumption is automatically satisfied.

The purpose of this paper is to prove an analogue of Theorem
\ref{T:classicalS} in the sub-Riemannian Heisenberg group $\HH$ (for
the relevant definitions we refer the reader to the next section).
The study of the Bernstein problem in this setting has received
increasing attention over the last decade. The existence of minimal
surfaces in sub-Riemannian spaces was established by two of us in
\cite{GarNh} by developing in such setting the methods of the
geometric measure theory. The study of minimal graphs in the
Heisenberg group was first approached by one of us in
\cite{Pauls:minimal}, by Cheng, Hwang, Malchiodi and Yang
\cite{CHMY} (who studied the problem in a more general class of
pseudohermitian $3$-manifolds), by three of us in
\cite{DGN:minimal}, and by two of us in \cite{GP}.

Henceforth in this paper, following a perhaps unfortunate but old
tradition, by \emph{minimal} we intend a $C^2$ surface $S\subset
\HH$ whose sub-Riemannian, or horizontal mean curvature $\mathcal H$
(see Proposition \ref{P:mc} below for its expression) vanishes
identically on $S$. In these initial investigations, a number of
nonplanar minimal graphs over the $xy$-plane are produced
(\cite{Pauls:minimal,CHMY,GP}) and indeed are classified (first in
\cite{CHMY}, with an alternate proof in \cite{GP}).  A prototypical
example is given by the surface $t=xy/2$ which is an entire minimal
graph over the $xy$-plane. However, this example and all other
entire minimal graphs over the $xy$-plane must have non empty
characteristic locus (this fact was proved independently in
\cite{CHMY} and \cite{GP}). We recall that the latter is defined as
the set of points of the surface at which the two bracket generating
vector fields $X_1, X_2$ become tangent to the surface itself.

In some of these same papers, new examples were discovered of entire
minimal graphs over some plane, but with an empty characteristic
locus. In \cite{GP}, two of us first produced infinitely many examples of
such graphs, one of which is given by
\begin{equation}\label{ce1}
x = y\ \tan(\tanh(t)).\end{equation}

Moreover, as announced in
\cite{CHMY} (this and many other examples are shown in more detail
in \cite{CH}), the surface
\begin{equation}\label{ce2}
x = y\ t
\end{equation}
is also noncharacteristic and minimal.  From the point of view of
the Bernstein problem, these examples would indicate a failure of
the property - there exists a rich reservoir of graphs over the
$xy$-plane which are minimal (although they have characteristic
points) and an equally rich reservoir of nonplanar noncharacteristic minimal
graphs over the $yt$-plane (or the $xt$-plane). In the positive
direction, the work \cite{GP} shows that graphs over vertical planes
must have a specific structure indicating some kind of rigidity (see
also \cite{CH} for other classification results).

In \cite{DGN:stable} the first three authors continued the
investigation into noncharacteristic graphs by asking a more refined
question:  are surfaces such as \eqref{ce1} or \eqref{ce2} local
minima? Just as in the classical case, sub-Riemannian minimal
surfaces are shown to merely be critical points of the relative area
functional (the so-called horizontal perimeter). Since this
functional is shown to lack the fundamental convexity property which
guarantees in the flat case that critical points are global
minimizers, the question of stability becomes central. It could
happen in fact that minimal surfaces such as \eqref{ce1},
\eqref{ce2} might fail to be locally area minimizing. Using a basic second
variation formula discovered in \cite{DGN:minimal}, in \cite{DGN:stable} the following surprising theorem is proved.

\begin{Thm}\label{T:DGNstable}  Let $\alpha
  >0,\beta \in \R$, then the surfaces
\[x\ =\ y \ (\alpha t+\beta)\ ,\ \ \;\; y\ =\ x\ (-\alpha t+\beta),\]
are unstable noncharacteristic entire minimal minimal graphs.
\end{Thm}

We emphasize that these surfaces are also global intrinsic graphs in the sense of \cite{FSSC2}, \cite{FSSC3}, see Definition
\ref{D:intgraph} below. We also note that
Theorem \ref{T:DGNstable} shows that an analogue of the Bernstein
property cannot hold unless we assume the surface be
noncharacteristic and stable.

The second variation formula in \cite{DGN:minimal} reduces to a
stability inequality of Hardy type on the surface. Another major
tool in the proof of Theorem \ref{T:DGNstable} is the reduction of
such Hardy type inequality to a one dimensional integral inequality
of Carleman-Wirtinger type which is confirmed by explicitly
constructing a variation which decreases perimeter. In \cite{DGNP},
we continued this line of investigation and provided a positive
answer to the following version of the Bernstein problem.

\begin{Thm}[\textbf{Bernstein Theorem 1, \cite{DGNP}}]\label{T:bern1}
In $\HH$ the only stable $C^2$ minimal entire graphs, with empty
characteristic locus, are the vertical planes \eqref{vp0}.
\end{Thm}

To illustrate the strategy behind this result, we recall a
definition from \cite{DGNP}.

\begin{Def}\label{D:gs}
We say that a $C^2$ surface $S\subset \HH$ is a \emph{graphical
strip} if there exist an interval $I\subset \R$, and $G \in C^2(I)$,
with $G'\geq 0$ on $I$, such that, after possibly a left-translation
and a rotation about the $t$-axis, then either
\begin{equation}\label{ceI}
S\ =\ \{(x,y,t)\in  \HH \mid  (y,t) \in \R \times I , x = y G(t)\},
\end{equation}
or
\begin{equation}\label{ceII}
S\ =\ \{(x,y,t)\in  \HH \mid  (x,t) \in \R \times I , y = - x
G(t)\}.
\end{equation}
If there exists $J\subset I$ such that $G'>0$ on $J$, then we call
$S$ a \emph{strict graphical strip}.
\end{Def}

It should be immediately clear to the reader that the surfaces in \eqref{ce1} or \eqref{ce2}
are examples of strict graphical strips in which one can take $J = I=\R$. Here is one of the two central results of \cite{DGNP}.

\begin{Thm}\label{T:DGNP1}
Any strict graphical strip is an unstable minimal surface with empty
characteristic locus. As a consequence, any minimal surface
containing a strict graphical strip is unstable.
\end{Thm}
The proof of Theorem \ref{T:DGNP1} involves, among other things, an
adaptation of the technique in \cite{DGN:stable} which leads to the
construction of an explicit variation along which the horizontal
perimeter strictly decreases. Our second main result in \cite{DGNP}
consists in proving, using the techniques of \cite{GP}, that every
noncharacteristic minimal graph over some plane which is not itself
a vertical plane
\begin{equation}\label{vp0}
\Pi_\gamma\ =\ \{(x,y,t)\in \HH\mid a x + b y = \gamma\},
\end{equation}
contains a strict graphical strip. Combining this result with
Theorem \ref{T:DGNP1}, we obtain Theorem \ref{T:bern1}.

Another approach to the sub-Riemannian Bernstein problem arises when
considering an intrinsic notion of graph. Observe that in flat
$\R^3$ a graph of the type $S = \{x = \phi(y,z)\mid(y,z)\in \Om\}$,
can be written as $S= \{(0,u,v) + \phi(u,v) e_1\mid (u,v)\in \Om\}$, where we have let $e_1 = (1,0,0)$.
Inspired by this observation Franchi, Serapioni and Serra Cassano
proposed the following notion of intrinsic graph adapted to the
non-Abelian group structure of $\HH$.

\begin{Def}\label{D:intgraph}
A $C^2$ surface $S$ is an \emph{intrinsic $X_1$-graph} if there exist a
domain $\Omega \subset \R^2_{uv}$ and $\phi \in C^2(\Omega)$, such
that $S=\{(0,u,v) \circ  \phi(u,v)e_1 | (u,v) \in \Omega\}$.
\end{Def}
We note that one basic consequence of this definition is that $S$
has empty characteristic locus. This follows from the fact that the
vector field $X_1$ is always transversal to the surface.
Interestingly, if we assume that $\Om$ be bounded, then the
horizontal perimeter of $S$ is given by the formula
\begin{equation}\label{igper} \mathscr{P}_H(S)\ =\ \int_\Om \sqrt{1
+ \mathcal B_\phi(\phi)^2}\ du\ dv,
\end{equation}
where we have denoted by $\mathcal B_\phi(f) = f_u + \phi f_v$ the
linearized Burger's operator. Notice that $\mathcal B_\phi(\phi) =
\phi_u + \phi \phi_v$ is the nonlinear inviscid Burger's operator.
Definition \ref{D:intgraph} was first introduced in \cite{FSSC2} and
developed further in \cite{FSSC3,AS,BSV,GS}.  In \cite{BSV}, Barone
Adesi, Serra Cassano and Vittone prove the following Bernstein theorem for
these types of graphs.

\begin{Thm}[\textbf{Bernstein Theorem 2, \cite{BSV}}]\label{T:bern2} The only $C^2$, stable minimal
entire intrinsic $X_1$-graphs are the vertical planes.
\end{Thm}

The proof of Theorem \ref{T:bern2} relies on a clever choice of
coordinates, suggested by the study of the characteristic curves of
the solutions of the minimal surface equation, which for an
intrinsic graph becomes

\begin{equation}\label{mseig}
\mathcal B_\phi(\mathcal B_\phi(\phi))\ =\ 0\ .
\end{equation}
Such change of coordinates allows the authors to reduce to a case which
can again be solved using the one dimensional reduction techniques used in
\cite{DGN:stable} to prove Theorem \ref{T:DGNstable}.

We are now in a position to discuss the results of this paper.
First, we introduce a definition which is related to Definition
\ref{D:gs} and that is suggested by the analysis of the double
Burger equation \eqref{mseig}. Suppose we are given some interval $J
= (-4\epsilon,4\epsilon)\subset \R$, $\epsilon >0$, and functions
$F, G, \sigma\in C^2(J)$ satisfying the condition
\begin{equation}\label{nd}
F'(s)^2 < 2 \sigma'(s) G'(s), \ \ \ \text{for every}\ s \in J.
\end{equation}

We note explicitly that \eqref{nd} implies, in particular, that
$\sigma'(s) G'(s)>0$ for every $s \in J$. If we consider the mapping
$\Psi: \R \times J \to \R^2$ from the $(u,s)$ to the $(u,v)$ plane
defined by $\Psi(u,s) = (u,v)$, where $v$ is defined by the equation
\begin{equation}\label{if}
v = v(u,s) = G(s) \frac{u^2}{2} + F(s) u + \sigma(s),
\end{equation}
then we see that the Jacobian determinant of $\Psi$ is given by
\begin{align}\label{jac}
\det J_\Psi(u,s) & =  \det \begin{pmatrix} 1 & 0 \\  G(s)u + F(s) &
 G'(s) \frac{u^2}{2} +  F'(s)u  + \sigma'(s)
\end{pmatrix}
\\
& = G'(s) \frac{u^2}{2} +  F'(s)u  + \sigma'(s). \notag\end{align}
Thanks to \eqref{nd} the Jacobian of $\Psi$ is always different from
zero. We emphasize at this moment that the continuity of the first
derivatives of the functions $F, G$ and $\sigma$, along with the
assumption \eqref{nd}, guarantee that, possibly restricting the
interval $J = (-4\epsilon,4\epsilon)$, we can always force the map
$\Psi$ to be globally one-to-one, hence a $C^2$ diffeomorphism of
$\R\times J$ onto its image $\Psi(\R\times J)$. We denote with
$\Psi^{-1}(u,v) = (u,s(u,v))$ the inverse $C^2$ diffeomorphism .
When we write $s(u,v)$ we mean the function specified by such
inverse diffeomorphism.

\begin{Def}\label{intdeltastrip}  Let $\epsilon >0$,
$J=(-4\epsilon,4\epsilon)$.  A $C^2$ surface $S\subset \HH$ is an \emph{intrinsic
graphical strip} on $J$ if there exist functions $F,G, \sigma \in
C^2(J)$ satisfying $(F')^2 \leq 2 \sigma' G'$ such that, if
\[\Omega= \Psi(\R\times J) = \{(u,v) | u \in \R, v=G(s) \frac{u^2}{2} + F(s) u + \sigma(s) \text{\; \;for
$s \in J$}\},\] then with $\phi\in C^2(\Omega)$ defined by
\[
\phi(u,v) = F(s(u,v)) +  u G(s(u,v)),
\]
we have
\[S = \{(0,u,v)\circ  (\phi(u,v),0,0)  | (u,v) \in \Omega\} = \{(\phi(u,v),u,v-\frac{u}{2} \phi(u,v))\mid (u,v)\in \Om\}.\]
We say that $S$ is a \emph{strict intrinsic graphical strip} on $J$
if $F, G, \sigma$ satisfy the strict inequality \eqref{nd}, and if
the map $\Psi:\R\times J \to \Om$ is globally one-to-one, hence a
diffeomorphism of $\R\times J$ onto $\Psi(\R\times J) = \Om$.
\end{Def}

\begin{Rem}\label{R:H-m}
A strict intrinsic graphical strip is necessarily a minimal surface.
To see this, we observe that the function $\phi$ in the above
definition satisfies \eqref{mseig}. The reader will find most of the
computations to achieve this in the proof of Corollary
\ref{T:varstrip}, see formula \eqref{sigmaderivs} below, and the
computations following that formula.
\end{Rem}

\begin{Rem}\label{R:G'>0}
In the case of a strict intrinsic graphical strip, without loss of
generality we can assume that $G'(s) > 0$ on $J$ (this property is
needed in the proof of Lemmas \ref{RHS}, \ref{LHS} and Theorem
\ref{I:unstable}).  We can justify this claim as follows. As
observed earlier, the condition \eqref{nd} implies $\sigma'(s)G'(s)
> 0$. Since $\sigma'(s)$ does not change sign, if $\sigma'(s) > 0$,
this forces $G'(s) > 0$. If instead we have $G'(s) < 0$ on $J$, we
replace $F,G,\sigma$ by $\tilde F(s) = F(-s)$, $\tilde G(s) =
G(-s)$, $ \tilde \sigma(s) = \sigma(-s)$.  The newly defined
functions satisfy \eqref{nd}. We also have $\tilde G'(s) > 0$.  Now
we take $\phi(u,v) = \tilde F(-s(u,v)) + u \tilde G(-s(u,v))$.  We
see that the surface parameterized by this new $\phi$ has the same
trace as the one with the original $\phi$.
\end{Rem}
\noindent

\begin{Rem} We emphasize here that any vertical plane such as \eqref{vp0} is an
intrinsic graphical strip, but not a strict intrinsic graphical
strip. One has in fact if $a\not= 0$, that $\phi(u,v) =
\frac{\gamma}{a} - \frac{b}{a} u$, so that $F(s) \equiv
\frac{\gamma}{a}$, $G(s) = - \frac{b}{a}$, $\sigma(s) \equiv 0$.
Therefore, $2 \sigma' G' - (F')^2 \equiv 0$.
\end{Rem}

Notice that, as a consequence of the smoothness hypothesis on $F,
G$, an intrinsic graphical strip is a surface of class $C^2$.
Definition \ref{intdeltastrip} takes advantage of the coordinates
introduced in \cite{BSV} discussed above, and the motivation behind
it will be explained in Section \ref{S:strips}. With this definition
in place and the second variation formula written in terms of
intrinsic graphical strips, in Section \ref{S:explicit} we use
techniques from \cite{DGN:stable} (and modifications from
\cite{DGNP}) to construct a variation on an intrinsic graphical
strip which decreases the horizontal perimeter, proving the
following basic result.

\begin{MThm}\label{I:unstable}
Let $S$ be a strict intrinsic graphical strip as in Definition
\ref{intdeltastrip}.
%Letting
%\[\psi =
%\frac{\chi{\sigma}\chi_k(u)}{1+uB'_k(\sigma)+\frac{u^2}{2}C'_k(\sigma)}\]
%as above and
There exists a $\psi\in C^2_0(S)$ such that
\[
\mathcal V^H_{II}(S,\psi X_1) < 0.
\]
As a consequence, $S$ is unstable.
\end{MThm}

The relevance of Theorem \ref{I:unstable} is in the following
theorem, which we prove in Section \ref{S:strips}.

\begin{MThm}\label{T:existstrip0} Every $C^2$ complete noncompact embedded
minimal surface without boundary with empty characteristic locus and
which is not itself a vertical plane contains a strict intrinsic
graphical strip.
\end{MThm}

Our proof of Theorem \ref{T:existstrip0} hinges on a close analysis
of the representation results of \cite{GP}. Theorems
\ref{I:unstable} and \ref{T:existstrip0} are the main novel
technical points of the present paper. From them, the following
theorem of Bernstein type will follow.

\begin{MThm}[of Bernstein type]\label{I:main}  The vertical planes are the only stable $C^2$ complete embedded
minimal surfaces in $\HH$ without boundary  and with empty
characteristic locus.
\end{MThm}

We note that Theorem \ref{I:main} is not contained in either of the
cited  Theorems \ref{T:bern1} or \ref{T:bern2}. For instance the
sub-Riemannian \emph{catenoids} in $\HH$ (the reader should note
that these surfaces are just the classical hyperboloids of
revolution)

\begin{equation}\label{cat0}
(t-a)^2\ =\ \frac{4}{b^2}\left(\frac{b}{4}(x^2+y^2)-1\right),\ \ \ \
a,b \in \R, b > 0,
\end{equation}
are complete embedded minimal surfaces with empty characteristic
locus which are not graphs on any plane, nor they are entire intrinsic
graphs. Theorem \ref{I:main} shows that such minimal surfaces are
unstable. These surfaces are a model of special interest. For this
reason, and also for making transparent to the reader our more
general constructions, we discuss them in detail in section
\ref{SS:catenoid}.

In closing, we note that the representation results of this paper
require that the surface be $C^2$:  the complete regularity theory
of minimal surfaces is currently an open problem which is being very
actively investigated.

This work was presented by the last named author at the ICM
Satellite Conference "Geometric Analysis and PDE's" in Naples,
Italy, September 2006, and by the third named author at the
Conference on Geometric Analysis and Applications, Univ. Illinois,
Urbana Champaign, July 2006. After its completion we were informed
of the preprint \cite{HRR} which addresses questions related to
those in this paper.

\section{\textbf{Definitions}}\label{S:def}

In this section we recall some definitions and known results which
will be needed in the paper. We recall that the Heisenberg group
$\Hn$ is the graded, nilpotent Lie group of step $2$ with underlying
manifold is $\mathbb C^{n}\times \R \cong \R^{2n+1}$, whose points
we indicate $g = (x,y,t)$, $g'=(x',y',t')$, etc., with non-Abelian
left-translation
\begin{equation}\label{Hn}
L_{g} (g') = g \circ g'\ =\ \left(x + x', y + y',t + t' +
\frac{1}{2} (x\cdot y' - x'\cdot y)\right),
\end{equation}
and non-isotropic dilations
\begin{equation}\label{Hn3}
\delta_\la (g) = (\la x, \la y, \la^2 t),\quad\quad\quad \la > 0.
\end{equation}

Here, and throughout the paper, we will use $v \cdot w$ to denote
the standard Euclidean inner product of two vectors $v$ and $w$ in
$\Rn$. The dilations \eqref{Hn3} provide a natural scaling
associated with the grading of the Heisenberg algebra $\mathfrak h_n
= V_1\oplus V_2$, where $V_1 = \mathbb R^{2n} \times \{0\}$, $V_2 =
\{0\}\times \R$. According to such scaling, elements of the
horizontal layer $V_1$ have degree one, whereas elements of the
vertical layer $V_2$ are assigned the degree two. The homogeneous
dimension associated with \eqref{Hn3} is $Q = 2n + 2$. We recall
that, identifying $\mathfrak h_n$ with $\mathbb R^{2n+1}$, we have
for the bracket
\[
[g,g']\ =\ (0,0, x\cdot y'- x'\cdot y).
\]
It is then clear that $[V_1,V_1] = V_2$, and that $V_2$ is the group
center.

Henceforth, we will focus on the first Heisenberg group $\HH$.
Applying the differential $(L_g)_*$ of \eqref{Hn} to the standard basis $\{\p_x,\p_y,\p_t\}$ of $\R^3$, we obtain
the three distinguished vector fields
\[X_1\ =\ (L_g)_*(\p_x)\ =\ \partial_x -\frac{y}{2} \partial_t\ , \;\; X_2\ =\ (L_g)_*(\p_y)\ =\ \partial_y
+\frac{x}{2} \partial_t\ ,\;\;T\ =\ (L_g)_*(\p_t)\ =\ \partial_t\ .\]
The horizontal bundle $H\HH$ is the subbundle of $T\HH$ whose fiber at a point $g\in \HH$ is given by
\[
H_g\ =\  span\{X_1(g),X_2(g)\} \ .
\]

We endow $\HH$ with a left-invariant Riemannian metric $\{g_{ij}\}$,
whose inner product we will denote by $<\cdot,\cdot>$, with respect
to which $\{X_1,X_2,T\}$ constitute an orthonormal basis. If
$S\subset \HH$ is a $C^2$ oriented surface we will indicate with
$\bN$ a (non-unit) Riemannian normal with respect to
$<\cdot,\cdot>$, and with $\n = \bN/|\bN|$ the corresponding Gauss
map. We will let
\begin{equation}\label{pqw}
p = <\bN,X_1>,\ \ \ q = <\bN,X_2>,\ \ \ W = \sqrt{p^2 + q^2},\ \ \
\omega\ =\ <\bN,T>.
\end{equation}

The characteristic locus of $S$ is the closed subset of $S$ defined by
\[\Sigma(S) = \{g \in S|W(g)=0\}.\]

We notice explicitly that $\Sigma(S) = \{g\in S\mid T_gS = H_g\}$. We also set on $S\setminus \Sigma(S)$
\begin{equation}\label{bars}
\pb = \frac{p}{W},\ \ \ \ \qb = \frac{q}{W},\ \ \ \ \ob =
\frac{\omega}{W}.
\end{equation}

\begin{Def}  Let $S\subset \HH$ be a $C^2$ oriented surface. A horizontal normal of $S$ is defined as
\[
\bN^H\ =\ p\ X_1\ +\ q\ X_2,\] whereas on $S\setminus \Sigma(S)$ the
horizontal Gauss map is defined as
\[\nuX = \frac{1}{W} \bN^H\ =\ \pb X_1 + \qb X_2\ .\]
\end{Def}

The horizontal perimeter measure of $S$ has the following form.

\begin{Pro}\label{P:Hper} Let $S\subset \HH$ be a $C^2$ oriented surface, then the horizontal perimeter of
  $S$ is
\[\mathscr{P}_H(S)=\int_S \sqrt{<\n,X_1>^2 + <\n,X_2>^2}\ d\sigma = \int_S \frac{W}{|\bN|} d\sigma,\]
where $d\sigma$ is the Riemannian surface area element associated to
$\ip$.
\end{Pro}

To investigate minimal surfaces, we recall the notion of horizontal
mean curvature $\mathcal H$ introduced in \cite{DGN:minimal},
\cite{Pauls:minimal}, \cite{GP}. Such notion is obtained by
projecting the horizontal Levi-Civita connection onto the so-called
horizontal tangent bundle $HTS = TS \cap H\HH$. If we assume, as we
may, that the Riemannian normal field on $S$, $\bN^H$, can be
extended to a neighborhood of $S$, and continuing to denote by $\pb,
\qb$ the quantities introduced in \eqref{bars} relative to such
extension, then it has been shown in the above cited references that
$\mathcal H$ can be computed by the following proposition.

\begin{Pro}\label{P:mc}
For $g \in S \minus \Sigma(S)$, the $H$-mean curvature of $S$ at $g$ is given by
\[\mathcal H(g)=X_1 \pb(g) + X_2\qb(g)\ .\]
\end{Pro}

For $g \in \Sigma(S)$, we define $\mathcal H(g)= \lim_{g' \ra g, g'
\in S \minus \Sigma(S)} \mathcal H(g')$, whenever the limit exists.
A surface $S$ is said to be \emph{minimal} if its horizontal mean
curvature vanishes identically.

It is now well known (\cite{CHMY,RR,GP,DGN:minimal,Pauls:minimal})
that critical points of the perimeter are characterized by having zero $H$-mean
curvature away from the characteristic locus.  We mention that recent
work of Cheng, Hwang and Yang (\cite{CHY}) and Ritor\'e and Rosales
(\cite{RR2}) have clarified the behavior of such critical points at
the characteristic locus.  However, since we will be restricting to
the category of noncharacteristic surfaces, we will not discuss these
results here.

%>>>>>>>>>>>>>>>>>>>>>>>>>>>>>>>>>>>>>>>>>>>>>>>>>

\section{\textbf{The second variation of the horizontal perimeter and the stability of minimal surfaces}}\label{S:var}

In this section, we recall the first and second variation of the
horizontal perimeter for intrinsic graphs. We mention that formulas
for the first and second variation of the horizontal perimeter have
been derived a number of times in various contexts
(\cite{RR,RR2,CHMY,DGN:minimal,GS,AS,BSV,HP2,HP4,BC}).

Let $S \subset \HH$ be an oriented $C^2$ surface with empty
characteristic locus, and consider vector fields $\mathcal X = a X_1
+ b X_2 + k T$, with $a, b , k\in C^2_0(\mathcal S)$. We define the
\emph{first variation} of the horizontal perimeter with respect to
the deformation of $S$,
\[
S^\lambda\ =\ S + \lambda \mathcal X,
\]
as
\[
\fv\ =\ \frac{d}{d\lambda}~ P_H(S^\lambda)\Bigl|_{\lambda = 0}.
\]
We say that $S$ is \emph{stationary} if $\fv = 0$, for every
$\mathcal X$. We define the \emph{second variation} of the
horizontal perimeter as
\[
\sv\ =\ \frac{d^2}{d\lambda^2}~ P_H(S^\lambda)\Bigl|_{\lambda = 0}.
\]
We say that $S$ is \emph{stable} is $\sv\geq 0$ for every $\mathcal
X$.

Henceforth, to simplify the formulas we introduce the following
notation
\begin{equation}\label{inner} F_{\mathcal X}\ \overset{def}{=}\ \pb a + \qb b +
\ob k\ =\ \frac{<\mathcal X,\bN>}{<\nuX,\bN>}.
\end{equation}

The following result was proved independently by several people in
various contexts,
see\cite{RR,RR2,CHMY,DGN:minimal,GS,AS,BSV,HP2,HP4,BC}.

 \begin{Thm}\label{T:variations}
Let $\mS\subset \HH$ be an oriented $C^2$ surface with empty
characteristic locus, then
\begin{equation}\label{fvH}
\mathcal V^H_I(S;\mathcal X)\ =\
 \int_{S}
\mathcal H\ \Fx\ d\sigma_H.
\end{equation}
In particular, $S$ is stationary if and only if it is minimal.
\end{Thm}

To state the next result we introduce a notation. Given the quantity
$\ob$ we let
\[
\mathcal A\ =\ -\ \del \ob.
\]

The following second variation formula was proved in
\cite{DGN:minimal}.

\begin{Thm}\label{T:svgeometric}
Let $\mS\subset \HH$ be a minimal surface with empty characteristic
locus, then
\[
\sv\ =\ \int_S \bigg\{|\del \Fx|^2\ +\ (2\mathcal A - \ob^2)
\Fx^2\bigg\} d\sigma_H.
\]
As a consequence, $S$ is stable if and only if for any $\mathcal X$
one has
\[
\int_S (\ob^2 - 2\mathcal A) \Fx^2 d\sigma_H\ \leq\ \int_S |\del
\Fx|^2\  d\sigma_H.
\]
\end{Thm}

The following result is Corollary 15.4 in \cite{DGN:minimal}. Let
$\phi:\Omega \subset \R^2_{(u,v)} \ra \R$ give an intrinsic
$X_1$-graph $S$, we recall the formula \eqref{igper} for the
horizontal perimeter of $S$.

\begin{Cor}\label{C:svig}
Let $S$ be a $C^2$ minimal, intrinsic $X_1$-graph, then for any
$\mathcal X$ one has
\[
\sv\ =\ \int_\Om \frac{\mathcal B_\phi(\Fx)^2}{\sqrt{1 + \mathcal
B_\phi(\phi)^2}}\ du dv\ -\ \int_\Om \frac{\phi_v^2 + 2 \mathcal
B_\phi(\phi_v)}{\sqrt{1 + \mathcal B_\phi(\phi)^2}}\ \Fx^2\ du dv,
\]
where $\Fx$ is as in \eqref{inner}.
\end{Cor}

We next derive the second variation formula for special deformations
of the intrinsic graph $S$. We consider compactly supported vector
fields on $S$ of the type $\mathcal X = \psi X_1$, where $\psi\in
C_0^2(S)$. For this family of deformations we obtain from Corollary
\ref{C:svig}.

\begin{Thm}\label{T:svig}
Let $S$ be a $C^2$ minimal, intrinsic $X_1$-graph, given by a
function $\phi:\Omega \subset \R^2_{(u,v)} \ra \R$, then for any
$\psi\in C^2_0(S)$ one has
\begin{align}\label{svig}
\mathcal V^H_{II}(S,\psi X_1)\ & =\ \int_\Om \frac{\mathcal
B_\phi(\psi)^2}{(1 + \mathcal B_\phi(\phi)^2)^{3/2}}\ du dv
\\
&    -\  \int_\Om \frac{\psi^2}{(1 + \mathcal B_\phi(\phi)^2)^{3/2}}
\bigg(2 \big(\mathcal B_\phi(\phi)\big)_v - \phi_v^2\bigg)\ du dv.
\notag\end{align}
\end{Thm}

\begin{Rem}\label{R:abuse}
In the statement of the above result the function $\psi\in
C_0^2(S)$. Slightly abusing the notation in the integral in the
right-hand side of \eqref{svig} we have continued to indicate with
$\psi$ the function in $C^2_0(\Om)$ obtained by composing the
original $\psi$ with the parametrization of the surface $S$
\[
\Om \ni (u,v)\ \longmapsto\ \left(\phi(u,v),u,v -
\frac{u}{2}\phi(u,v)\right).
\]
\end{Rem}

\begin{proof}[\textbf{Proof}]
We note that with $\mathcal X = \psi X_1$, we have $a = \psi$, $b =
k = 0$. We also recall, see \cite{DGN:minimal}, that for an
intrinsic $X_1$-graph one has
\[
\pb\ =\ \frac{1}{\sqrt{1 + \mathcal B_\phi(\phi)^2}}\ ,\ \ \ \qb\ =\
-\ \frac{\mathcal B_\phi(\phi)}{\sqrt{1 + \mathcal B_\phi(\phi)^2}},
\]
and therefore from \eqref{inner} one has
\begin{equation}\label{Fig}
\Fx\ =\ \frac{\psi}{\sqrt{1 + \mathcal B_\phi(\phi)^2}}.
\end{equation}

From this formula a simple computation gives
\[
\mathcal B_\phi(\Fx)\ =\ \frac{\mathcal B_\phi(\psi)}{\sqrt{1 +
\mathcal B_\phi(\phi)^2}}\
 -\ \frac{\mathcal B_\phi(\phi) \mathcal B_\phi\big(\mathcal B_\phi(\phi)\big)}{(1 + \mathcal B_\phi(\phi)^2)^{3/2}}.
 \]

We now recall that the minimality of $S$ is equivalent to $\phi$
being a solution of the double Burger equation
\[
\mathcal B_\phi(\mathcal B_\phi(\phi))\ =\ 0.
\]
We thus conclude that
\begin{equation}\label{BFig}
\mathcal B_\phi(\Fx)\ =\ \frac{\mathcal B_\phi(\psi)}{\sqrt{1 +
\mathcal B_\phi(\phi)^2}}.
\end{equation}

Using \eqref{Fig} and the identity
\[
\big(\mathcal B_\phi(\phi)\big)_v\ -\ \mathcal B_\phi(\phi_v)\ =\
\phi_v^2,
\]
we thus obtain
\[
-\ \int_\Om \frac{\phi_v^2 + 2 \mathcal B_\phi(\phi_v)}{\sqrt{1 +
\mathcal B_\phi(\phi)^2}}\ \Fx^2\ du dv\  =\ -\  \int_\Om
\frac{\psi^2}{(1 + \mathcal B_\phi(\phi)^2)^{3/2}} \bigg(2
\big(\mathcal B_\phi(\phi)\big)_v - \phi_v^2\bigg)\ du dv.
\]

On the other hand, \eqref{BFig} gives
\[
\int_\Om \frac{\mathcal B_\phi(\Fx)^2}{\sqrt{1 + \mathcal
B_\phi(\phi)^2}}\ du dv\ =\ \int_\Om \frac{\mathcal
B_\phi(\psi)^2}{(1 + \mathcal B_\phi(\phi)^2)^{3/2}}\ du dv.
\]

Combining the last two equations we reach the desired conclusion.

\end{proof}

Next, we apply Theorem \ref{T:svig} to the case of a strict
intrinsic graphical strip as in Definition \ref{intdeltastrip}. We
recall the diffeomorphism $\Psi:\R\times J\to \Om = \Psi(\R\times J)
\subset \R^2_{u,v}$ given by $\Psi(u,s) = (u,v) =
(u,\frac{u^2}{2}G(s)+ F(s)u + \sigma(s))$, see \eqref{if}. As
before, in the statement of the next result given a function
$\psi\in C_0^2(S)$ slightly abusing the notation we will write
$\psi\in C^2_0(\Om)$. What we mean by this is the composition of the
original $\psi$ with the parametrization of the surface $S$
\[
\Om \ni (u,v)\ \longmapsto\ \left(\phi(u,v),u,v -
\frac{u}{2}\phi(u,v)\right)
\]
provided in Definition \ref{intdeltastrip}.

\begin{Cor}\label{T:varstrip}  Let $S$ be a strict intrinsic
graphical strip defined by functions $F, G, \sigma\in C^2(J)$ and
$\phi(u,v) = F(s(u,v)) + u G(s(u,v))$, as in Definition
\ref{intdeltastrip}. One has for any $\psi\in C^2_0(S)$,
\begin{align}\label{T:strip}
\mathcal V^H_{II}(S,\psi X_1)&\ =\ \int_{\R \times J}
\bigg(\left(\frac{\partial}{\partial u} (\psi\circ\Psi)(u,s)
\right)^2 \frac{G'(s) \frac{u^2}{2} + F'(s)u +
\sigma'(s)}{(1+G(s)^2)^{\frac{3}{2}}}\\
\notag &\qquad\qquad\qquad\qquad
\ +\
\frac{(\psi\circ\Psi)(u,s)^2}{(1+G(s)^2)^{\frac{3}{2}}}\,
\frac{F'(s)^2-2\sigma'(s)G'(s)}{G'(s) \frac{u^2}{2} +  F'(s)u  +
\sigma'(s)} \bigg) du ds, \notag
\end{align}
where we have indicated with $\Psi:\R \times J \to \Om$ the
diffeomorphism defined by \eqref{if}.
\end{Cor}

\begin{proof}[\textbf{Proof}]
We note that the proof of this theorem is similar to that of
equation (5.12) of \cite{BSV}. Since every strict intrinsic
graphical strip is an intrinsic $X_1$-graph, we can apply the second
variation formula \eqref{svig} in Theorem \ref{T:svig}. In this
formula we want to use the global diffeomorphism $\Psi:\R \times J
\to \Om$ to convert the integral on $\Om$ to an integral on
$\R\times J$. By \eqref{jac}
\begin{align*}
\det J_\Psi(u,s) & =  \det \begin{pmatrix} 1 & 0
\\
v_u & v_s\end{pmatrix} = \det \begin{pmatrix} 1 & 0 \\
G(s)u + F(s) &
 G'(s) \frac{u^2}{2} +  F'(s)u  + \sigma'(s)
\end{pmatrix}
\\
& = G'(s) \frac{u^2}{2} +  F'(s)u  + \sigma'(s).
\end{align*}
We emphasize that since we are assuming that $S$ is a strict
graphical strip, then \eqref{nd} is in force, and therefore the
Jacobian of $\Psi$ is always different from zero. Recall that we are
also assuming that $\Psi$ is globally one-to-one. The Inverse
Function Theorem gives at every point $(u,v) = \Psi(u,s)$
\[
J_{\Psi^{-1}}(u,v) = \begin{pmatrix} 1 & 0 \\
- \frac{G(s)u + F(s)}{G'(s) \frac{u^2}{2} +  F'(s)u  + \sigma'(s)} &
 \frac{1}{G'(s) \frac{u^2}{2} +  F'(s)u  + \sigma'(s)}
\end{pmatrix}.
\]
We thus have
\begin{equation}\label{sigmaderivs}
%\begin{split}
s_u = - \frac{G(s)u + F(s)}{G'(s) \frac{u^2}{2} +  F'(s)u  +
\sigma'(s)},\;\; \ \ s_v = \frac{1}{G'(s) \frac{u^2}{2} +  F'(s)u  +
\sigma'(s)}.
%\end{split}
\end{equation}
Using \eqref{sigmaderivs} and the assumption that $\phi(u,v)= F(s) +
u G(s)$, we thus find
\begin{equation*}
\begin{split}
\mathcal B_\phi(\phi) &= \phi_u + \phi \phi_v =
G(s)+(G'(s)u+F'(s))s_u+\phi(G'(s)
u+F'(s))s_v\\
&= G(s) - \frac{(F'(s)+uG'(s))(F(s)+uG(s))}{G'(s) \frac{u^2}{2} +  F'(s)u  + \sigma'(s)} +
\frac{(F'(s)+uG'(s))(F(s)+uG(s))}{G'(s) \frac{u^2}{2} +  F'(s)u  + \sigma'(s)}\\
&=G(s).
\end{split}
\end{equation*}
This gives,
\[
 \big(\mathcal B_\phi(\phi)\big)_v = G'(s) s_v = \frac{G'(s)}{G'(s) \frac{u^2}{2} +  F'(s)u  + \sigma'(s)},
 \]
 \[
 (\phi_v)^2= \left(F'(s) + u G'(s)\right)^2 s_v^2 =  \left(\frac{F'(s) + u G'(s)}{G'(s) \frac{u^2}{2} +  F'(s)u  + \sigma'(s)}\right)^2.
 \]
Combining these formulas yields
\[
2\big(\mathcal B_\phi(\phi)\big)_v - \phi_v^2 =
\frac{2\sigma'(s)G'(s)-F'(s)^2}{G'(s) \frac{u^2}{2} +  F'(s)u  +
\sigma'(s)}\ .
\]
Substituting this into the second integral in the right-hand side of
\eqref{svig} gives
\[ \mathcal V^H_{II}(S,\psi X_1)\ =\ \int_\Omega
\frac{1}{(1+G(s)^2)^{\frac{3}{2}}} \left ( (\mathcal B_\phi(\psi)^2
+ \psi^2
    \left(
      \frac{F'(s)^2-2\sigma'(s)G'(s)}{G'(s) \frac{u^2}{2} +  F'(s)u  + \sigma'(s)} \right )\right ) \; du \;dv\ .\]
Now, to complete the proof, we make the change of variable $(u,v) =
\Psi(u,s)$, with $(u,s)\in \R \times J$. The Jacobian of such
diffeomorphism is given by \eqref{jac} which gives \[ du dv =
\left(G'(s) \frac{u^2}{2} +  F'(s)u  + \sigma'(s)\right) du ds.
\] Observe furthermore that
\[
\mathcal B_\phi(\psi) =  \psi_u+\phi\psi_v = \psi_u + (F + G
u)\psi_v = \psi_u + v_u \psi_v =  \frac{\partial}{\partial u}
\psi(u,v(u,s)) = \frac{\p }{\p u} (\psi \circ \Psi)(u,s).
\]

Thus, we conclude that
\begin{align*}
\mathcal V^H_{II}(S,\psi X_1)\ & = \int_{\R \times J}
\bigg(\left(\frac{\p }{\p u} (\psi \circ \Psi)(u,s)) \right)^2
\frac{G'(s) \frac{u^2}{2} + F'(s)u +
\sigma'(s)}{(1+G(\sigma)^2)^{\frac{3}{2}}}
\\
&\qquad\qquad\qquad\qquad
 +\ \frac{(\psi\circ
\Psi)(u,s)^2}{(1+G(s)^2)^{\frac{3}{2}}}
\frac{F'(s)^2-2\sigma'(s)G'(s)}{G'(s) \frac{u^2}{2} +  F'(s)u  +
\sigma'(s)} \bigg) du ds, \notag
\end{align*}
which proves \eqref{T:strip}.

\end{proof}

\section{\textbf{Proof of Theorem \ref{I:unstable}: Strict intrinsic graphical strips are unstable}}\label{S:explicit}

In this section using the techniques of \cite{DGN:stable} and the
modifications of \cite{DGNP}, we construct a variation which
strictly decreases the horizontal area of a strict intrinsic
graphical strip (that is, we find a test function $\psi$ for which
$\mathcal V^H_{II}(S,\psi X_1) < 0$.  To construct such a $\psi$
we start by constructing a sequence $\psi_k$.  We will show that
for large enough $k$, we have $\mathcal V^H_{II}(S,\psi_k X_1) < 0$.
This proves that such surfaces
are unstable, thus establishing Theorem \ref{I:unstable}.

For any given $\delta>0$, we fix a function $\chi \in
C^\infty_0(\R)$ so that $0 \le \chi(s) \le 1, \chi(s)=1$ for $|s|
\le \delta, \chi(s)=0$ for $|s|\ge 2\delta$, and $|\chi'|\le
C=C(\delta)$.  For each $k \in \mathbb{N}$, we let
$\chi_k(s)=\chi(s/k)$ and hence
\begin{itemize}
\item $\chi_k(s) =0$ for $|s|\ge 2\delta k$
\item $\chi_k(s)=1$ for $|s| \le \delta k$
\item $|\chi_k'(s)| \le C/k$
\end{itemize}
Next, fix a function $\zeta \in C^\infty_0(\mathbb{R})$ with $\zeta
\ge 0$, $supp(\zeta)=[-1,1]$ and $\int_\mathbb{R} \zeta \; ds =1$.
Letting, $\zeta_k(s)=k\zeta(ks)$, we have that $supp({\zeta}_k) =
[-1/k,1/k]$ and $\int_\R {\zeta}_k(s)\; ds =1$. Let $F$, $G$ and
$\sigma$ be the functions in Definition \ref{intdeltastrip} with
\begin{equation}\label{E:str-cond}
F'(s)^2 - 2\sigma'(s)G'(s) < 0\ \quad\ \ \ s\in J.
\end{equation}
As we have mentioned in the introduction, without loss of generality
we assume that $G', \sigma'>0$ in $J$. We define $F_k = F\star
{\zeta}_k$, $G_k = G\star \zeta_k$, $\sigma_k = \sigma \star
\zeta_k$.  Since $F$, $G$ and $\sigma$ are continuous on $J$.
Shrinking $J$ slightly if necessary, we may assume that they are
uniformly continuous on $\bar J$. Therefore $F_k\to F$, $F'_k\to
F'$, $G_k\to G$, $G'_k\to G'$, $\sigma_k \to \sigma$ and
$\sigma_k'\to \sigma'$ uniformly on $\bar J$.  The condition
\eqref{E:str-cond} now carries over to $F_k, G_k, \sigma_k$, that
is, there is a positive integer $k_o$ such that if $k > k_o$
(relabeling the sequence if necessary, we take $k_o = 1$) then for
every $s\in J$, $F_k'(s)^2 - 2\sigma_k'(s) G_k'(s) < 0$. The left
hand side of this inequality is precisely the discriminant of the
quadratic expression in the variable $u$:
\[
G'_k(s) \frac{u^2}{2} +  F'_k(s)u  +
\sigma'_k(s)\ .
\]
Since the discriminant is strictly negative, $G'_k(s) \frac{u^2}{2}+
F'_k(s)u + \sigma'_k(s)$ never vanishes for $u\in\R$ and $s\in J$.
Next, we construct a sequence of test functions $\psi_k$ to be used
in the formula \eqref{T:strip}.  We let

\begin{equation}\label{E:tests}
\psi_k(u,s) \overset{def}{=} \frac{\chi(s)\chi_k(u)}{\left(G'_k(s)
\frac{u^2}{2} +  F'_k(s)u  + \sigma'_k(s)\right)^{\frac{1}{2}}}\ ,
\end{equation}
We note that $\psi_k\in C^\infty_0(\R\times J)$ due to the above
considerations. With $\psi_k$ in hand, we analyze $\mathcal
V^H_{II}(S,\psi_k X_1)$.    Before proceeding to the computations,
we remark that the function $\psi$ in \eqref{T:strip} is defined on
$\Om = \Psi(\R\times J)$.  Our $\psi_k$'s have been already defined
on the $(u,s)$ space, that is on $\R\times J$.  Therefore,
occurrences of $\psi\circ\Psi$ in \eqref{T:strip} will be replaced
by $\psi_k$ in the proof of the subsequent two lemmas.  We start
with the second term in the right hand side of \eqref{T:strip}.

\begin{Lem}\label{RHS} We have
\begin{equation*}
\begin{split}
\lim_{k \ra \infty} \int_{\R \times J} \frac{\psi_k(u,s)^2}{(1+G(s)^2)^{\frac{3}{2}}}
&\,\frac{F'(s)^2-2\sigma'(s)G'(s)}{G'(s) \frac{u^2}{2} +  F'(s)u  +
\sigma'(s)}\, du\;ds \\
&\ =\ -2\pi\int_J
\frac{\chi(s)^2}{(1+G(s)^2)^\frac{3}{2}}
\,\frac{G'(s)}{(2\sigma'(s)G'(s) - F'(s)^2)^{\frac{1}{2}}} \, ds
\end{split}
\end{equation*}
\end{Lem}

\begin{proof}[\textbf{Proof}]
Substituting the quantity $\psi\circ\Psi$ with $\psi_k$ in the second term
of the right hand side of \eqref{T:strip} and recalling the definition of
$\psi_k$ we have

\begin{align}\label{tmp0}
& \lim_{k \ra \infty} \int_{\R \times J}
\frac{\psi_k(u,s)^2}{(1+G(s)^2)^{\frac{3}{2}}}\,
\frac{F'(s)^2-2\sigma'(s)G'(s)}{G'(s)\frac{u^2}{2} + F'(s) u +
\sigma'(s)}\, du \; ds  \\ &\ =\ \lim_{k \ra \infty} \int_J
\chi(s)^2\frac{F'(s)^2 -
2\sigma'(s)G'(s)}{(1+G(s)^2)^{\frac{3}{2}}} \notag\\
& \times \left(\int_{\R} \frac{\chi_k(u)^2}{(G_k'(s)\frac{u^2}{2} +
F_k'(s) u + \sigma_k'(s)) (G'(s)\frac{u^2}{2} + F'(s) u +
\sigma'(s))} \, du\right) ds \notag\\ &\ =\ \int_J
\chi(s)^2\frac{F'(s)^2-2\sigma'(s)G'(s)}{(1+G(s)^2)^{\frac{3}{2}}}
\left(\int_\R \frac{1}{(G'(s)\frac{u^2}{2} + F'(s) u +
\sigma'(s))^2}\, du \right ) ds.\notag
\end{align}
In the above, we have used the fact that since for each $u\in\R$,
\[
G_k'(s)\frac{u^2}{2} + F_k'(s) u + \sigma_k'(s) \longrightarrow
G'(s)\frac{u^2}{2} + F'(s) u + \sigma'(s) \quad\text{as } k \to \infty
\]
uniformly for $s\in \bar J$, and the latter quantity never vanishes,
we have
\[
\frac{1}{2} |G'(s)\frac{u^2}{2} + F'(s) u + \sigma'(s) |
\ <\
|G_k'(s)\frac{u^2}{2} + F_k'(s) u + \sigma_k'(s)|
\ <\
2 |G'(s)\frac{u^2}{2} + F'(s) u + \sigma'(s)|\ .
\]
Hence, Lebesgue dominated convergence theorem allows taking the
limit inside the integral. Next, we want to compute the integral
\[
\int_\R \frac{1}{(G'(s)\frac{u^2}{2} + F'(s) u + \sigma'(s))^2}\,
du.
\]
Using standard integration techniques we obtain
\[
\int \frac{1}{(Au^2 + Bu + C)^2} \,du
\ =\
\frac{2Au + B}{(4AC-B^2)(Au^2+Bu+C)}\ +\
\frac{4A}{(4AC - B^2)^\frac{3}{2}}\, arctan\left(\frac{2Au+B}{\sqrt{4AC - B^2}}\right)\ .
\]
This implies if $A > 0$
\[
\int_{\R} \frac{1}{(Au^2 + Bu + C)^2} \,du
\ =\
\frac{4\pi A}{(4AC - B^2)^\frac{3}{2}}\ .
\]
Since we have that $G'(s) > 0$, letting $A = G'(s)/2$, $B = F'(s)$ and $C = \sigma'(s)$ we have

\begin{equation}\label{E:tmp1}
\int_\R \frac{1}{(G'(s)\frac{u^2}{2} + F'(s) u + \sigma'(s))^2}\, du
\ =\ 2\,\pi \frac{G'(s)}{(2\sigma'(s)G'(s) - F'(s)^2)^\frac{3}{2}}\ .
\end{equation}
Substituting \eqref{E:tmp1} in \eqref{tmp0} we reach the desired
conclusion.

\end{proof}

Now we turn to the first term in the right hand side of \eqref{T:strip}.

\begin{Lem}\label{LHS}  We have
\begin{align*}
& \lim_{k \ra \infty}
\int_{\R \times J}
\left(\left(\frac{\partial
  \psi(u,s)}{\partial u} \right )^2
\frac{G'(s)\frac{u^2}{2} + F'(s) u + \sigma'(s)}{(1+G(s)^2)^{\frac{3}{2}}}\right
)\,du \; ds \\
& \qquad\qquad
\ = \frac{\pi}{2}\int_J
\frac{\chi(s)^2}{(1+G(s)^2)^{\frac{3}{2}}}\,
\frac{G'(s)}{(2\sigma'(s)G'(s)-F'(s)^2)^{\frac{1}{2}}} \, ds
\end{align*}
\end{Lem}

\begin{proof}[\textbf{Proof}]
Again, we closely follow the development in \cite{DGNP}. By
recalling \eqref{E:tests} we first obtain
\[
\frac{\partial \psi_k}{\partial u}(u,s) \ =\
\frac{\chi(s)}{2}\left(\frac{2\chi_k'(u) Q_k(u,s) - \chi_k(u)
D_k(u,s)}{Q_k(u,s)^\frac{3}{2}}\right),
\]
where we have let
\[
Q_k(u,s) = G_k'(s)\frac{u^2}{2} + F_k'(s) u + \sigma_k'(s)
\quad\text{and}\quad D_k(u,s) = u G_k'(s) + F_k'(s).
\]
For the computations that follow, it is convenient to also let
\[
Q(u,s) = G'(s)\frac{u^2}{2} + F'(s) u + \sigma'(s)
\quad\text{and}\quad D(u,s) = \frac{\partial}{\partial u} Q(u,s) = u
G'(s) + F'(s).
\]
It follows that
\[
\left(\frac{\partial \psi_k}{\partial u}(u,s)\right)^2  =
\chi(s)^2\left(\frac{\chi_k'(u)^2}{Q_k(u,s)} -
\frac{1}{2}(\chi_k(u)^2)'\frac{D_k(u,s)}{Q_k(u,s)^2} +
\frac{1}{4}\chi_k(u)^2 \frac{D_k(u,s)^2}{Q_k(u,s)^3}\right).
\]
Substituting the quantity $\psi\circ\Psi$ in the first term of the
right hand side of \eqref{T:strip}, and using the above expression
for $\psi_{k,u}$, we have
\begin{equation*}
\begin{split}
\int_{\R\times J} \left(\frac{\partial
  \psi_k(u,s)}{\partial u} \right )^2
\frac{G'(s)\frac{u^2}{2} + F'(s) u +
\sigma'(s)}{(1+G(s)^2)^{\frac{3}{2}}} \, du\;ds & = \int_J
\frac{\chi(s)^2}{(1+G'(s)^2)^{\frac{3}{2}}}
\left(\,\fbox{1}+\fbox{2}+\fbox{3}\, \right )\,ds
\end{split}
\end{equation*}
where,
\begin{align*}
\fbox{1} =  \int_\R \chi_k'(u)^2 &\frac{Q(u,s)}{Q_k(u,s)}\, du,
\qquad
\fbox{2} = -\frac{1}{2}\int_\R (\chi_k^2(u))'Q(u,s)\frac{D_k(u,s)}{Q_k(u,s)^2}\, du, \\
& \fbox{3} = \frac{1}{4} \int_\R \chi_k(u)^2 Q(u,s)
\frac{D_k(u,s)^2}{Q_k(u,s)^3} \, du.
\end{align*}
Since $|\chi_k'(u)| \le \frac{C}{k}$, by Lebesgue dominated
convergence theorem we have
\begin{equation}\label{LHS0}
\lim_{k \ra \infty} \fbox{1} =0
\end{equation}
In addition, since $D_k(u,s) \ra D(u,s), Q_k(u,s) \ra Q(u,s)$, and
$\chi_k(s) \ra 1$ when $k \ra \infty$, we obtain
\begin{align}\label{LHS3}
\lim_{k \ra \infty} \fbox{3}
&\ =\ \frac{1}{4} \int_\R \frac{D(u,s)^2}{Q(u,s)^2}\, du
\ =\
-\,\frac{1}{4} \int_\R \frac{\partial}{\partial u} Q(u,s)\,\frac{\partial}{\partial u}\left(\frac{1}{Q(u,s)}\right)\,du \\
\notag
&\ =\ \frac{1}{4}\int_\R\frac{\partial^2 Q(u,s)}{\partial u^2}\,\frac{1}{Q(u,s)}\,du
\ =\ \frac{1}{4}\int_\R \frac{G'(s)}{G'(s)\frac{u^2}{2} + F'(s) u + \sigma'(s)}\,du \\
\notag
&\ =\ \frac{\pi\,G'(s)}{(2\sigma'(s)G'(s) - F'(s)^2)^\frac{1}{2}}\ .
\end{align}
The third equality above is obtained by integration by parts whereas
in the last equality, we have used the fact that $G'(s)>0$ and
standard calculus techniques. Now we turn to the quantity
$\fbox{2}$.

\begin{align}\label{LHS4}
\lim_{k \ra \infty}\fbox{2}
&\ =\ -\lim_{k \ra \infty} \frac{1}{2}
\int_\R \left(\chi_k(u)^2)\right)' Q(u,s)\frac{D_k(u,s)}{Q_k(u,s)}\,du \\
&  =
- \lim_{k \ra \infty}\frac{1}{2}\int_\R \chi_k(u)^2 \frac{\partial}{\partial u}\left(\frac{Q(u,s)\,D_k(u,s)}{Q_k(u,s)^2}\right)\,du \\
\notag\\ & = - \lim_{k \ra \infty}\frac{1}{2}\int_\R \chi_k(u)^2
\bigg(\frac{Q_u(u,s)\,D_k(u,s)}{Q_k(u,s)^2}
\notag\\
& + \frac{Q(u,s)\,D_{k,u}(u,s)}{Q_k(u,s)^2} -
2\,\frac{Q(u,s)\,D_k(u,s)\,Q_{k,u}(u,s)}{Q_k(u,s)^3}\bigg)\,du
\notag\\ & = - \frac{1}{2}\int_\R \frac{Q_u(u,s)\,D(u,s)}{Q(u,s)^2}
+ \frac{D_u(u,s)}{Q(u,s)} - 2\,\frac{D(u,s)Q_u(u,s)}{Q(u,s)^2}\,du
\notag\\ & =
-\frac{1}{2}\int_\R \frac{G'(s)}{Q(u,s)}\,du -
\frac{1}{2}\int_\R \frac{\partial}{\partial u} Q(u,s)\,\frac{\partial}{\partial u}\left(\frac{1}{Q(u,s)}\right)\,du \notag\\
\notag & = -\frac{1}{2}\int_\R \frac{G'(s)}{Q(u,s)}\,du  +
\frac{1}{2}\int_\R \frac{Q_{uu}(u,s)}{Q(u,s)}\,du  = 0, \notag
\end{align}
since $Q_{uu}(u,s) = G'(s)$. Combining \eqref{LHS0}, \eqref{LHS3}
and \eqref{LHS4}, we obtain the desired conclusion.

\end{proof}

Combining \eqref{T:strip} with Lemmas \ref{RHS} and \ref{LHS} we can
now prove Theorem \ref{I:unstable} in the introduction.

\vspace{.2in}

\noindent
\begin{proof}[\textbf{Proof of Theorem \ref{I:unstable}.}]
Let $\psi_k$ be the function constructed in \eqref{E:tests} and
consider $\psi_k\circ \Psi^{-1}\in C^2_0(\Om)$, where $\Psi$ is the
diffeomorphism in \eqref{if}. If we lift this function to the
surface, and by abuse of notation we continue to indicate with
$\psi_k$ such lifted function, we obtain a function in $C^2_0(S)$.
From Corollary \ref{T:varstrip}, Lemmas \ref{RHS}, \ref{LHS} and the
fact that $G'(s) > 0$ on $J$ we deduce that
\[
\lim_{k \ra \infty}\mathcal V^H_{II}(S,(\psi_k X_1)
 =
\left(\frac{\pi}{2}-2\pi\right)\int_J
\frac{\chi(s)^2}{(1+G(s)^2)^\frac{3}{2}}
\,\frac{G'(s)}{(2\sigma'(s)G'(s) - F'(s)^2)^{\frac{1}{2}}} \, ds
 < 0.
\]
Therefore, for large enough $k$ we have $\mathcal V^H_{II}(S,\psi_k
X_1) < 0$. This completes the proof.

\end{proof}

%>>>>>>>>>>>>>>>>>>>>>>>>>>>>>>>>>>>>>>>>>>>>>>>>>>

\section{\textbf{Proof of Theorem \ref{T:existstrip0}: Existence of strict intrinsic graphical strips}}\label{S:strips}

The main objective of this section is to establish the crucial
Theorem \ref{T:existstrip0} in the introduction. The proof of this
result will be accomplished in several steps. Before we turn to the
general discussion it will be helpful for the understanding of
Definition \ref{intdeltastrip} to analyze directly the situation of
the surfaces introduced in \eqref{cat0}.

\subsection{The sub-Riemannian catenoid is unstable}\label{SS:catenoid}

In what follows we illustrate the construction of a strict intrinsic
graphical strip for the hyperboloids of revolution in $\HH$
described by \eqref{cat0}. This is an interesting example of a
complete embedded minimal surface in $\HH$ which has empty
characteristic locus and which is neither a graph over any plane,
nor an intrinsic graph in the sense of \cite{FSSC2}, \cite{FSSC3}.
Such surface should be considered as the sub-Riemannian analogue of
the \emph{catenoid} in the classical theory of minimal surfaces. We
emphasize that \eqref{cat0} does not contain any strict graphical
strip in the sense of \cite{DGNP}, and therefore the results in that
paper do not apply to it. Instead, as a consequence of the following
calculations and Theorem \ref{I:unstable} we are able to conclude
that the surface \eqref{cat0} is unstable. To fix the ideas we will
focus on the case $a=0, b=4$, in which case we have from
\eqref{cat0}
\begin{equation}\label{cat}
t^2 - \frac{1}{4}\left((x^2+y^2)-1\right).
\end{equation}
A local parametrization of $S$ as a ruled surface is given by
\begin{equation}\label{theta}
 \theta(r,s) = \left(r \sin s + \cos s, r \cos s - \sin s,
\frac{r}{2}\right),\ \ \ r\in \R, -\pi< s < \pi.
\end{equation}
Clearly, if we consider the open set $U = \R \times (-\pi,\pi)$,
then $\theta(U)$ does not cover the whole catenoid, but this fact in
inconsequential for what follows. We now consider the projection
mapping $\Pi:\R^3 \to \R^2\times \{0\}$ given by
\[
\Pi(x,y,t) = (0,y,t+\frac{xy}{2}).
\]
We thus have
\begin{align*}
\Pi(\theta(U)) & = \left(0,r \cos s - \sin s,\frac{r}{2} + \frac{(r
\sin s + \cos s)(r \cos s - \sin s)}{2}\right).
\\
& = \left(0,r \cos s - \sin s, \frac{r^2}{2} \sin s \cos s  + r
\cos^2 s - \frac{\sin s \cos s}{2}\right)
\end{align*}
We now define a mapping from the $(r,s)$ to the $(u,s)$ plane by
setting \[ \Lambda(r,s) = (r \cos s - \sin s,s). \] With
$\epsilon\in (0,\pi/4)$ to be chosen later, and
\[
U_\epsilon = \R \times (-\epsilon,\epsilon),
\]
it is clear that $\Lambda$ is a $C^\infty$ diffeomorphism of
$U_\epsilon$ onto its image $\Lambda(U_\epsilon)$. Notice that,
thanks to the fact that $1<\sec s<\sqrt 2$ for $-\epsilon <
s<\epsilon$, we have $\Lambda(U_\epsilon) = U_\epsilon$. Let us
notice that the inverse diffeomorphism is given by
\[ (r,s) = \Lambda^{-1}(u,s) = \left(\frac{u+\sin s}{\cos
s},s\right) = (u \sec s + \tan s, s).
\]

Next, we define a mapping from the $(r,s)$ to the $(u,v)$ plane by
setting
\[
\Phi(r,s) = (u,v)
\]
with
\begin{equation}\label{1}
\begin{cases}
u = r \cos s - \sin s,
\\
v = \frac{r^2}{2} \sin s \cos s + r \cos^2 s - \frac{\sin s \cos
s}{2}.
\end{cases}
\end{equation}
We want to show that $\Phi$ is a diffeomorphism onto its image. To
see this we take the composition $\Psi = \Phi \circ \Lambda^{-1} :
U_\epsilon \to \R^2$, which maps the $(u,s)$ to the $(u,v)$ plane.
We obtain
\[
(u,v) = \Psi(u,s) = \Phi(\Lambda^{-1}(u,s)) = \left(u,G(s)
\frac{u^2}{2} + F(s) u + \sigma(s)\right),
\]
where
\[
\begin{cases}
G(s) =  \tan s,
\\
F(s) = \sec s,
\\
\sigma(s) =  \frac{\tan s}{2}.
\end{cases}
\]
Let us observe that the determinant of the Jacobian of $\Psi(u,s)$
at any point $(u,s)\in U_\epsilon$ is given by
\[
G'(s)\frac{u^2}{2} + F'(s)u + \sigma'(s) = \frac{\sec^2 s}{2}
\left[u^2 + 2 \sin s\ u + 1 \right].
\]
Since for the quadratic expression within the square brackets we
have
\[
\Delta = \sin^2 s - 1 < 0,
\]
it is clear that such determinant never vanishes. We next show that
$\Psi$ is globally  one-to-one on $U_\epsilon$ provided that
$\epsilon>0$ is chosen sufficiently small. Suppose by contradiction
that $(u,s), (u',s')\in U_\epsilon$, $(u,s) \not= (u',s')$, and
$\Psi(u,s) = \Psi(u',s')$. It cannot be $u\not= u'$ (since then
$\Psi(u,s) \not= \Psi(u',s')$). We can thus suppose that $s\not=
s'$, but $u= u'$. Since $\tan s$ is strictly increasing, $s\not= s'$
implies $G(s)\not= G(s')$. But then we must have
\begin{equation}\label{2}
u^2 + 2 \frac{F(s) - F(s')}{G(s) - G(s')} u + 1 = 0.
\end{equation}
We would like to show that there exists $0<\epsilon <\pi/4$ such
that for every $s,s'\in (-\epsilon,\epsilon)$, with $s\not= s'$, one
has
\begin{equation}\label{3}
\left(\frac{F(s) - F(s')}{G(s) - G(s')}\right)^2 < 1.
\end{equation}
If this were the case then we would reach a contradiction since this
implies that the equation \eqref{2} has no real solutions. Now
\eqref{3} is equivalent to
\begin{equation}\label{4}
\left(\frac{\sec s - \sec s'}{\tan s - \tan s'}\right)^2 < 1,
\end{equation}
for every $s,s'\in (-\epsilon,\epsilon)$, with $s\not= s'$. Without
restriction we can assume $s<s'$, otherwise we reverse their role.
Using the mean value theorem we find that for some $\xi, \xi'\in
(s,s')\subset (-\epsilon,\epsilon)$
\[
\frac{\sec s - \sec s'}{\tan s - \tan s'} = \frac{\sec \xi \tan
\xi}{1 + \tan^2 \xi'} \to 0,\ \ \ \text{as}\ \epsilon \to 0^+.
\]
Therefore, we can achieve \eqref{3} provided that $\epsilon > 0$ is
sufficiently small. Having fixed $\epsilon$ in such a way, the map
$\Psi : U_\epsilon \to \R^2$ defines a $C^\infty$ diffeomorphism
from the $(u,s)$ plane onto its image $V_\epsilon \overset{def}{=}
\Psi(U_\epsilon)$, which is an open set of the $(u,v)$ plane. We now
claim that there exists $\delta = \delta(\epsilon)>0$ such that
\begin{equation}\label{strip}
\Om \overset{def}{=} \R \times (-\delta,\delta) \subset V_\epsilon.
\end{equation}
To prove \eqref{strip} it suffices to show that, as $s$ ranges over
the interval $(-\epsilon,\epsilon)$ the $v$-coordinate of the
vertices of the parabolas $v = v(u,s) = \frac{\tan s}{2} u^2 + \sec
s\ u + \frac{\tan s}{2}$ are uniformly bounded away from zero. Let
us notice that the line $s = 0$ in the $(u,s)$ plane is mapped to
the line $v=u$ of the $(u,v)$ plane. For $s\not= 0$ the $v$
coordinate of the vertex of the parabola is given by
\[
v(s) = - \frac{\sec^2 s(1-\sin^2 s)}{2 \tan s} = - \frac{\cot s}{2}.
\] Now on the interval $0<s<\epsilon$ we have $v(s)\to -\infty$ as
$s\to 0^+$, whereas on $(-\epsilon,0)$ we have $v(s) \to + \infty$
as $s\to 0^-$. Since $\cot s$ is strictly decreasing on
$(-\epsilon,\epsilon)$, we conclude that if we take
\[
\delta = \delta(\epsilon) = \frac{\cot \epsilon}{2},
\]
then \eqref{strip} is verified. Since the composition of
diffeomorphisms is a diffeomorphism as well, we conclude that
\[
\Phi \overset{def}{=} \Psi \circ \Lambda : U_\epsilon \to \Om
\subset \R^2_{u,v}
\] is also a diffeomorphism.  At this point, using the inverse
diffeomorphism $\Phi^{-1}: \Om \to U_\epsilon$, we define
\[
\phi(u,v) = \theta_1(\Phi^{-1}(u,v)),\ \ \ (u,v)\in \Om,
\]
where $\theta_1(r,s) = r \sin s + \cos s$ is the first component of
the map $\theta$ in \eqref{theta}. Notice that
\[
\phi(u,v) = F(s(u,v)) + G(s(u,v)) u,
\]
where $(r(u,v),s(u,v))$ is the inverse diffeomorphism of \eqref{1}.

With this definition of $\phi$ we now see that portion of the
catenoid which is parametrized by $\theta$ on the open set
$U_\epsilon = \R\times (-\epsilon,\epsilon)$ is in fact given as the
$X_1$-graph
\[
\left(\phi(u,v),u, v - \frac{u}{2} \phi(u,v)\right),
\]
for $(u,v)\in \Om$. Finally, let us notice that such piece of the
surface is a strict intrinsic graphical strip in the sense of
Definition \ref{intdeltastrip} since the condition
\[
F'(s)^2 < 2 G'(s) \sigma'(s),\ \ \ \ s\in (-\epsilon,\epsilon),
\]
is verified.

\subsection{Proof of Theorem \ref{T:existstrip0}}\label{SS:MT}

The above analysis should allow the reader a clear understanding of
the motivation behind the Definition \ref{intdeltastrip} of strict
intrinsic graphical strip. Our next objective is proving that,
similarly to the sub-Riemannian catenoid, every complete minimal
surface without boundary and with empty characteristic locus
contains a strict intrinsic graphical strip, unless the surface is a
vertical plane. In this general case the construction of the strict
graphical strip is more difficult. Our approach hinges on the
following basic representation theorem for minimal surfaces which is
a consequence of the results in \cite{GP}, and which has already
proved crucial in \cite{DGNP}.

\begin{Thm}\label{T:thma}
Let $S$ be a $C^2$ complete embedded non-characteristic minimal
surface without boundary and assume that it is not a vertical plane.
Let $g_0\in S$ be a point admitting a neighborhood (in $S$) that may
be written as a graph over the plane $t=0$. There exist a
neighborhood $U$ of $g_0$, an interval $J$, and functions $h_0 \in
C^2(J)$, $\gamma \in C^3(J,\R^2)$, with $|\gamma'(s)| = 1$ for $s\in
J$,  such that $U$ is parameterized by $\mathscr{L}: \R \times J \to
\mathbb{H}$
\begin{equation}\label{seedrep}
\mathscr{L}(r,s)\ =\
 \left (\gamma(s)+r{(\gamma')}^\perp(s),h_0(s)-\frac{r}{2}\gamma(s) \cdot
   \gamma'(s)\right )
\end{equation}
for $s \in J, r \in \R$.  Moreover, with $W_0(s)=h_0'(s)+\frac{1}{2}\gamma' \cdot
\gamma^\perp(s)$ and $\kappa(s)=\gamma'' \cdot (\gamma')^\perp$, we have that
\begin{equation}\label{cl}
 1-2W_0(s)\kappa(s)\ < 0\ ,\ \ \ s\in J\ .
\end{equation}
\end{Thm}

The proof of Theorem \ref{T:thma} will be presented after Corollary
\ref{foliation} below. We first develop some preparatory results.

\begin{Lem}\label{L:ruled}
Let $D\subset \R^2$ be an open set, $g\in C^2(D)$, and consider the
$C^2$ map $G:D \rightarrow \mathbb{H}^1$ given by
$G(x,y)=(x,y,g(x,y))$. Suppose that $S=G(D)$ is a non-characteristic
minimal surface. Then $S$ is foliated by horizontal straight lines
which are the integral curves of $\n_H^\perp = \qb X_1 - \pb X_2$.
\end{Lem}

\begin{proof}[\textbf{Proof}]
Writing $S$ as the level set $\phi(x,y,t) =
g(x,y)-t=0$ we have that
\[\n_H = \pb\;X_1+\qb\;X_2,\]
where
\[\pb=\frac{X_1 \phi}{\sqrt{(X_1 \phi)^2+(X_2 \phi)^2}},\; \qb=\frac{X_2 \phi}{\sqrt{(X_1 \phi)^2+(X_2 \phi)^2}}.\]
The reader should keep in mind here that
\begin{equation}\label{pq}
p = X_1 \phi = X_1g +\frac{y}{2} = g_x + \frac{y}{2},\ \ \ q = X_2
\phi = X_2g -\frac{x}{2} = g_y - \frac{x}{2}.
\end{equation}
We emphasize that the assumption that $S$ be non-characteristic is
equivalent to
\[
W = \sqrt{(X_1 \phi)^2+(X_2 \phi)^2 }\neq 0\ \ \ \ \text{on}\ D.
\] By Proposition \ref{P:mc} we see that assumption that $S$ be
minimal reads
\[X_1\pb + X_2\qb\ =\ 0,\]
which, using the fact that $\pb, \qb$ are independent of $t$, is
equivalent to
\begin{equation}\label{MSEp}
 div V =\pb_x+\qb_y\ =\ 0.
\end{equation}
Here, we view $V=\pb\partial_x+\qb\partial_y$ as a vector field on
$D$ and $div$ is the Euclidean divergence. We now claim that if
$c(s)=(c_1(s),c_2(s))\subset D$ is an integral curve in $D$ of
$V^\perp = \qb \partial_x - \pb
\partial_y$, then $C(s)=(c_1(s),c_2(s),g(c_1(s),c_2(s)))$ must be an
integral curve of $\n_H^\perp$ on $S$.  To see this suppose that
$c'(s) = V^\perp(c(s))$, which means $c_1'(s) = \qb(c(s)), c_2'(s) =
- \pb(c(s))$. Now from these equations and from \eqref{pq} one has
\begin{equation}\label{Tcomponent}
c_1'\left(g_x(c) + \frac{c_2}{2}\right) + c_2'\left(g_y(c) -
\frac{c_1}{2}\right) = \qb(c) p(c) - \pb(c) q(c) = \frac{q(c) p(c) -
p(c) q(c)}{W} = 0,
\end{equation}
where for simplicity we have omitted the variable $s$ when writing
$c$ instead of $c(s)$. Now,
\[
C'(s) = (c_1'(s),c_2'(s),g_x(c_1(s),c_2(s))c_1'(s) +
g_y(c_1(s),c_2(s)) c_2'(s)).
\]
Using the formula
\begin{equation}\label{passage}
a X_1 + b X_2 + c T = \left(a,b,c+ \frac{bx - ay}{2}\right),
\end{equation}
which allows to pass from the standard representation in terms of
the Cartesian coordinates in $\HH$ to that with respect to the
orthonormal basis $\{X_1,X_2,T\}$, we find
\begin{align}\label{horiz}
C'(s)&= c_1'(s) X_1(c(s)) +c_2'(s) X_2(c(s)) \\
& +\left (\nabla g(c_1(s),c_2(s)) \cdot
(c_1'(s),c_2'(s))+\frac{c_1'(s)c_2(s)-c_1(s)c_2'(s)}{2} \right ) T.
\notag\end{align}

From \eqref{Tcomponent} we conclude that the component of $C'(s)$
with respect to $T$ is identically equal to zero, and therefore
\[
C'(s) = c_1'(s) X_1(c(s)) + c_2'(s) X_2(c(s)) = \qb(c(s)) X_1(c(s))
- \pb(c(s))X_2(c(s)) = \n_H^\perp(c(s)).
\]
This proves the claim.

Since $V$ is a unit vector field, we have that
\begin{equation}\label{uniteqs}
\pb\ \pb_x+ \qb\ \qb_x = \pb\ \pb_y + \qb\ \qb_y = 0.
\end{equation}
Combining equations \eqref{uniteqs} and \eqref{MSEp}, we conclude
that if $c(s) = (c_1(s),c_2(s))$ is an integral curve of $V^\perp$,
with $c(0) = z = (x,y)\in D$, then \[ \frac{d}{ds} V^\perp(c(s))
\equiv 0\ ,
\] and therefore $V^\perp(c(s)) \equiv V^\perp(c(0)) = V^\perp(z)$. It follows that
\[
c(t) = z + s V^\perp(z), \] i.e., $c(s)$ is a segment of straight
line in $D^2$ passing through $z$. If we write $c(t)=(x + as,y+
bs)$, then
\[
C(s)=(x+as,y+bs,g(x+as,y+bs)),
\]
and so
\[
C'(s) = (a,b,a g_x + b g_y).
\]
At this point we note that the vanishing of the $T$ component in
\eqref{Tcomponent} now implies that
\[
\frac{d}{ds} g(x+as,y+bs) = g_x a + g_y b = -
\frac{c_1'(s)c_2(s)-c_1(s)c_2'(s)}{2} = \frac{bx-ay}{2}.
\]
This gives
\[
g(x+as,y + bs) = g(z) + \frac{bx-ay}{2}s,
\]
and therefore,
\[
C(s)=\left(x+as,y+bs,g(z) + \frac{bx-ay}{2}s\right),
\]
i.e., $C(s)$ is a straight line segment in $\HH$.

\end{proof}

\begin{Lem}\label{L:lemma2}  Suppose $S$
be a $C^2$ non-characteristic minimal surface such that no open
subset of $S$ may be written as a graph over the $xy$-plane. Then,
$S$ is a piece of a vertical plane and, hence, is foliated by
horizontal straight lines which are the integral curves of
$\n_H^\perp$.
\end{Lem}
%\noindent {\em Proof: }

\begin{proof}[\textbf{Proof}]
Let $(x_0,y_0,t_0)\in S$ and let $U\subset \HH$ be an open
neighborhood  of $(x_0,y_0,t_0)$ such that $S\cap U = \{(x,y,t)\in
U\mid \phi(x,y,t) = 0\}$ for a $\phi\in C^2(U)$ having $\nabla
\phi\not= 0$  in $U$.  By the assumption that no open subset of $S$
may be written as a graph over the $xy$-plane, we see that it must
be $\phi_t=0$ in $U$. Then, $\phi(x,y,t)=\phi_0(x,y)$ in $U$ and
therefore $S\cap U$ is a portion of  a ruled surface over a curve
$c$ in the $xy$-plane. Furthermore, due to the special structure of
$\phi$ one easily recognizes that the assumption that $S$ be
$H$-minimal now translates into the fact that $\phi_0$ must satisfy
the classical minimal surface equation
\begin{equation*} div \left(\frac{\nabla
\phi_0}{\sqrt{1+|\nabla \phi_0|^2}}\right) = 0\ , \end{equation*} on
the open set $\tilde U = \pi(U)\subset \R^2$, where $\pi(x,y,t) =
(x,y)$. This equation is in fact equivalent to
\begin{equation}\label{lmse} (1 + \phi_{0,y}^2) \phi_{0,xx} - 2
\phi_{0,x} \phi_{0,y} \phi_{0,xy} + (1 + \phi_{0,x}^2) \phi_{0,yy}\
=\ 0\ .
\end{equation}
Since $\nabla \phi_0 \not= 0$ in $U$, by the Implicit Function
Theorem, we may locally describe the curve $c$ by either $y=g(x)$ or
$x=f(y)$. In the former case, we have $\phi_0(x,y) = y - g(x)$, and
thus \eqref{lmse} implies that $g''=0$. We conclude that there
exists an open set $V\subset \HH$ containing $(x_0,y_0,t_0)$ such
that $S\cap V$ is a piece of a vertical plane. The second case leads
to the same conclusion. By the assumption that $S$ be $C^2$ we now
conclude that if for two such different open sets $V_1, V_2$ one has
$V_1\cap V_2\not= \varnothing$, then the two corresponding portions
of planes $S\cap V_1$ and $S\cap V_2$ must be part of the same
plane. This completes the proof.

\end{proof}

In the next lemma we combine into a single result the two different
situations considered in Lemmas \ref{L:ruled} and \ref{L:lemma2}.

\begin{Lem}\label{locfol} Let $S$ be a $C^2$ minimal
surface in $\mathbb{H}^1$ with empty characteristic locus, and let
$p$ be a point in the interior of $S$ (in the relative topology).
Then, there exists a neighborhood $\Delta$ of $p$ in $S$ which is
foliated by horizontal straight line segments which are integral
curves of $\n_H^\perp$.
\end{Lem}

\begin{proof}[\textbf{Proof}]
For every $p\in \overset{\circ}{S}$, there exists an open set
$U\subset \HH$ and a $\phi\in C^2(U)$ such that $\nabla \phi\not= 0$
in $U$ and $\Sigma = S\cap U = \{(x,y,t)\in U\mid \phi(x,y,t) =
0\}$. Let $S_1=\{(x,y,t)\in \Sigma | \phi_t(x,y,t) \neq 0\}$,
$S_2=\{(x,y,t)\in \Sigma | \phi_t(x,y,t) = 0\}$. Notice that, either
$\phi_t\equiv 0$ on $\Sigma$ and in such case $S_2 = \Sigma$ is a
vertical cylinder over a curve in the $xy$ plane, or there exists an
open set $V\subset \HH$ such that $S_2\cap V$ is a $C^1$ curve in
$\HH$.  In the former case we can invoke Lemma \ref{L:lemma2} to
conclude that $\Delta = \Sigma$ is foliated by horizontal straight
line segments which are integral curves of $\n_H^\perp$. We are thus
left with the case in which $S_1\not= \varnothing$. By shrinking
$\Sigma$ if necessary we can assume that $\Sigma=S_1 \cup S_2$,
where $S_2$ is a $C^1$ curve.

In our arguments, we consider integral curves of $\n_H^\perp$
passing through points on the surface $S$.  To make this notion
precise, we recall that as $S$ is a $C^2$ submanifold of
$\mathbb{H}^1 =\mathbb{R}^3$, every point $p \in S$ is contained in
a coordinate chart $i: D \subset \R^2 \ra S$ with $i \in C^2(D)$.
For any $C^1$ vector field, $U_0$, defined on $i(D)$, the integral
curve of $U_0$ passing through $q \in i(D)$ is simply $i(\gamma)$
where $\gamma \subset D$ is a solution to the initial value problem:
\begin{equation*}
\begin{split}
\gamma'(t)&=i^{-1}_*(U_0)(\gamma(t))\\
\gamma(0)&=i^{-1}(q).
\end{split}
\end{equation*}
Direct calculation then shows that
\[
\frac{d}{dt}i(\gamma)=i_*i^{-1}_*U_0(\gamma(t))=U_0(i(\gamma(t))),
\]
and $i(\gamma(0))=i(i^{-1}(q))=q$.  As $U_0$ (and hence $i^* U_0$)
is $C^1$, the standard theorems concerning solutions to ODE apply to
the integral curves of $U_0$ on $S$.  In particular, we may conclude
that given $q \in S$, there exists (at least for a short time) a
unique integral curve of $U_0$.  Similarly, we conclude that
integral curves of $U_0$ on $S$ have continuous dependence on
parameters.

By Lemma \ref{L:ruled}, each point in $S_1$ is contained in a
neighborhood which is foliated by straight line segments which are
integral curves of $\nu_H^\perp$. Thus, those portions of integral
curves of $\n_H^\perp$ contained in $S_1$ are at least piecewise
linear.  By the fact that $\n_H^\perp$ is $C^1$ and the uniqueness
of solutions to ode's, we must have that these portions of integral
curves are straight lines. We may extend each such  line segment
maximally within $S_1$. If a limit point of a maximally extended
line segment were in $S_1$, we could apply Lemma \ref{L:ruled} to
extend it further, violating the assumption that we had extended
maximally. Thus we conclude that the limit points of the line
segment are in $\partial S_1 \cup S_2$.

Consider $p \in S_2$ and let $c$ be the integral curve of
$\n_H^\perp$ with $c(0)=p$.  Let $B_\epsilon$ be the metric ball of
radius $\epsilon$ centered at $p$ and $c_\epsilon = c \cap
B_\epsilon$.  Then, there exists an $\epsilon >0$ sufficiently small
so that one of the following possibilities occurs:
\begin{enumerate}
\item  $c_\epsilon \cap S_2$ is closed and has no interior;
  \item $c_\epsilon \cap S_2$ is closed with nonempty interior and $p$ is in the
  interior;
 \item $c_\epsilon \cap S_2$ is closed with nonempty interior and $p$ is contained in the boundary of the interior of $c_\epsilon \cap S_2$.
\end{enumerate}

In the first case, $c_\epsilon \cap S_1$ is open and dense in
$c_\epsilon$.  By Lemma \ref{L:ruled}, every point in $c_\epsilon
\cap S_1$ is contained in an open line segment which is a subset of
$c_\epsilon$.  As $c_\epsilon \cap S_2$ is closed and is contained
in the boundary of $c_\epsilon \cap S_1$, we conclude that
$c_\epsilon$ is piecewise linear.  By the smoothness of $\n_h^\perp$
and the uniqueness of solutions to ODE, we conclude $c_\epsilon$ is
a single straight line segment.

In the second case, we may shrink $\epsilon$ so that $c_\epsilon
\cap S_2=c_\epsilon$ and $S_2$ divides $B_\epsilon \cap S$ into
exactly two pieces $N_1,N_2$.  We next show that if $q \in N_1$ is
contained in a line segment, $L \subset N_1$, which reaches the
boundary of $N_1$ then the length of $L$ is at least
$2(\epsilon-\delta)$ where $\delta$ is the Euclidean distance from
$p$ to $q$.  Observe that the endpoints of $L$ can not be in $S_2$.
If one were in $S_2$, then by the uniqueness of solutions of ODE, we
conclude that $L$ and $S_2$ coincide.  This contradicts our
assumption that $q \not \in S_2$.  Thus, $L$ must be a line segment
in $B_\epsilon$ which has both its boundary points in $\partial
B_\epsilon$.  By construction, the Euclidean distance from $p$ to
the endpoints of $L$ is $\epsilon$.  Denoting the Euclidean distance
from $p$ to $q$ by $\delta$, the triangle inequality implies that
the length of $L$ is at least $2(\epsilon -\delta)$.

Let $ q_i \in N_1$ be a sequence of points converging to $p$ and let
$L_i$ be the maximal line segment which is the integral curve of
$\n_H^\perp$ through $q_i$ which is contained in $N_1$.  By the
continuous dependence on parameters of the solutions to an ODE and
the fact the $\n_H^\perp$ is $C^1$, we know $L=\lim_{i \ra \infty}
L_i$ exists and is an integral curve of $\n_H^\perp$ passing through
$p$.  Moreover, since $L$ is the limit of lines segments each of
whose lengths are bounded below by $2(\epsilon - \delta_i)$ (where
$\delta_i$ is the Euclidean distance from $p$ to $q_i$), we conclude
$L$ is a line segment of length at least $2\epsilon$.  Note that so
far, we have shown that every point in $S_1$ and every point in
$S_2$ that fall in cases one and two are contained in an open line
segment which is an integral curve of $\n_H^\perp$.

We are left with points of $S_2$ which fall into the third category.
The collection of such points in $S_2$ is, by construction, closed
and has empty interior.  Thus, $c_\epsilon$ contains an open dense
set of points that are either in $S_1$ or fall in one of the first
two cases above.  For each such points, Lemma \ref{L:ruled} or the
discussion of the first two cases yields an open line segment
containing the point which is a subset of $c_\epsilon$.  Thus, as in
the argument for case one, $c_\epsilon$ is piecewise linear and, by
the smoothness of $\n_H^\perp$, must be a single straight line
segment.

Using the arguments above for points in $S_2$ and Lemma
\ref{L:ruled} for points in $S_1$, we see that integral curve of
$\n_H^\perp$ through any point contains a line segment through that
point.  Thus, all such integral curves are piecewise linear and, by
the smoothness of $\n_H^\perp$, must be straight lines.  Combining
all of these arguments shows that $\Sigma$ is foliated by straight
line segments which are integral curves of $\n_H^\perp$.

\end{proof}

\begin{Cor}\label{foliation}  Let $S$ be a $C^2$ connected complete
  non-characteristic minimal surface without boundary in
  $\mathbb{H}^1$.  Then, $S$ is foliated by horizontal straight lines
  which are integral curves of $\n_H^\perp$.
\end{Cor}
\noindent

\begin{proof}[\textbf{Proof}]
Since $S$ is assumed to have no boundary, for any $p \in S$ Lemma
\ref{locfol} implies that there exists an open neighborhood of $p$
which is foliated by such straight line segments.  By the smoothness
of $\n_H^\perp$, we have that $S$ itself is foliated by such
straight line segments.  It remains to show that the entirety of
each line is contained in $S$.

Let $L:(-\epsilon,\epsilon) \ra S$ be a line segment with $L(0)=p
\in S$ and $L'(t)=\n_H^\perp(L(t))$ and let $\tilde{L}:\mathbb{R}
\ra \mathbb{H}^1$ be the full line containing $L$ so that
$\tilde{L}(t)=L(t)$ for $t \in(-\epsilon,\epsilon)$.  Let \[ I =
\{t\in \R\mid  \tilde{L}(t)\in S\}\ . \] By construction, $I$ is not
empty since $0\in I$. Let $t_i \in I$ be a sequence of parameters so
that $t_i \ra t_\infty$ where $t_\infty $ is a limit point of $I$.
By completeness of $S$, we must have that $\lim_{i \ra
\infty}\tilde{L}(t_i) = \tilde{L}(t_\infty)$ is an element of $S$.
Thus, $I$ is closed as it must contain all of its limit points. But,
$I$ is open as well.  To see this, consider $p=\tilde{L}(t)$ for a
fixed $t \in I$.  As $\partial S = \varnothing$, $p$ is in the
interior of $S$ and so, by Lemma \ref{locfol}, $p$ is contained in a
neighborhood which is foliated by straight lines which are integral
curves of $\n_H^\perp$.  Thus, $I$ must contain an open neighborhood
of $t$. Since $I$ is both open and closed, we conclude that $I =
\mathbb{R}$ and that $\tilde{L}(\R) \subset S$.

\end{proof}

\noindent
\begin{proof}[\textbf{Proof of Theorem \ref{T:thma}}]
By Corollary \ref{foliation}, we have that $S$ is foliated by
horizontal straight lines which are integral curves of $\n_H^\perp$.
Let $O$ be an open neighborhood of $g_0$ which may be written as a
graph $(x,y,h(x,y))$ with $h \in C^2$. Consider a unit tangential
vector field, $\mathcal{W}$, defined on $O$ which is perpendicular
(with respect to the fixed Riemannian metric) to $\n_H^\perp$.  Let
$(\gamma_1(s),\gamma_2(s),h_0(s))$ be an integral curve of
$\mathcal{W}$ so that $\gamma(0)=g_0$ with domain $J$.  Note that
$\gamma_1,\gamma_2,h_0 \in C^2(J)$ as $\n_H^\perp$ is $C^1$. Let $N$
be the collection of lines in the foliation which pass through point
of the curve $(\gamma_1(J),\gamma_2(J),h_0(J))$.  Then, since for a
fixed $s_0 \in J$, we have from \eqref{passage}
\begin{align*}
\mathscr{L}_{s_0}'(r)& =(\gamma_2'(s_0),-\gamma_1'(s_0),
-\frac{1}{2}(\gamma_1(s_0),\gamma_2(s_0))\cdot(\gamma_1'(s_0),\gamma_2'(s_0)))
\\
& = \gamma_2'(s_0)\;X_1-\gamma_1'(s_0)\;X_2 = \n_H^\perp,
\end{align*}
the line of the foliation passing through
$(\gamma_1(s_0),\gamma_2(s_0),h_0(s_0))$ is given by
\[\mathscr{L}_{s_0}(r)=(\gamma_1(s_0)+r\gamma_2'(s_0),\gamma_2(s_0)-r\gamma_1'(s_0),h_0(s_0)-\frac{r}{2}(\gamma_1(s_0),\gamma_2(s_0))\cdot(\gamma_1'(s_0),\gamma_2'(s_0)))\]
Thus, $N$ may be parametrized by $\mathscr{L}:\R \times J \ra
\mathbb{H}^1$ given by
\begin{equation}\label{T3.4-eq1}
\mathscr{L}(r,s) = (
\gamma_1(s)+r\gamma_2'(s),\gamma_2(s)-r\gamma_1'(s),h_0(s)-\frac{r}{2}\gamma(s)\cdot\gamma'(s)).
\end{equation}
It remains to show that $\gamma=(\gamma_1,\gamma_2) \in C^3(J)$.  As $O$ is a graph over a region $\bar{O}$ of the xy-plane,  $\bar{\mathscr{L}}(r_0,s)=(\gamma_1(s)+r\gamma_2'(s),\gamma_2(s)-r\gamma_1'(s))$ parametrizes  a subset of $\bar{O}$ with $s \in J, r \in (-\epsilon,\epsilon)$ for $\epsilon$ sufficiently small.  Under this parametrization, $V=\pb \;\partial_x+\qb \;\partial_y = \gamma_1'(s) \;\partial_x+\gamma_2'(s)\; \partial_y$.  We
first observe that, for a fixed $r=r_0$, the curve $s\to
\bar{\mathscr{L}}(r_0,s)$ coincides with the integral curve of
$V$ through the point $\bar{\mathscr{L}}(r_0,0)$ on their
mutual domain of definition (we may assume, by shrinking $J$ if
necessary, that $J$ is the mutual domain of definition). To see
this, note that the definition of $\bar{\mathscr{L}}$ gives
\[
\bar{\mathscr L}_s(r,s) = (\gamma_1'(s) + r\gamma_2''(s),\gamma_2'(s) - r\gamma_1''(s))
\]
This implies
\[\langle \bar{\mathscr{L}}_s(r_0,s),V^\perp \rangle= \gamma_2'\gamma_1'+r\gamma_2'\gamma_2''-\gamma_1'\gamma_2'+r \gamma_1''\gamma_1'=0.\]
The last equality follows from the fact that $|\gamma'|\equiv 1$ on
$J$.   Let $\bar{c}\subset \mathbb{R}^2$ be the integral curve of $V$ passing through
$\bar{\mathscr{L}}(r_0,0)$. We note that $\bar{c}$ is parameterized by
arc-length and, to avoid confusion, we will denote its parameter by
$\xi$.  Since $V$ is $C^1$, we have that $\bar{c} \in C^2(\xi)$.
Moreover, since $O$ is given by $(x,y,h(x,y))$ with $h \in C^2$, we
see that $c(\xi) = h(\bar{c}(\xi))$ is $C^2(\xi)$ as well.

To facilitate our computations, we note that
\[
|\bar{\mathscr{L}}_s(r_0,s)|=|1-r_0 \kappa(s)|. \] This can be
verified as follows. Recalling that $|\gamma'| = 1$ and that $\kappa
= \gamma_1'' \gamma_2' - \gamma_2'' \gamma_1'$, one easily obtains
\[
|\bar{\mathscr{L}}_s(r_0,s)|^2 = 1 - 2 r \kappa(s) +
r^2(\gamma_1''(s)^2 + \gamma_2''(s)^2).
\]
Now, some elementary considerations give
\[
\kappa(s)^2 = ((\gamma_1''(s)^2 + \gamma_2''(s)^2)|\gamma'(s)|^2 - 2
(\gamma'(s)\cdot \gamma''(s))^2 = (\gamma_1''(s)^2 +
\gamma_2''(s)^2),
\]
and this implies the desired conclusion. Let now $\kappa_0=
\underset{s\in J}{\sup} |\kappa(s)|$. If $\kappa_0=0$, then $\gamma$
is a line segment and hence $\gamma$ is certainly $C^3$. Assuming
$\kappa_0
>0$, we pick $r_0< \min \{\kappa_0^{-1},\epsilon\}$ which implies that
$|\bar{\mathscr{L}}_s(r_0,s)|=|1-r_0 \kappa(s)|=1-r_0\kappa(s)$.  We note
that $\xi$ is differentiable in $s$ as $\bar{c}(\xi)$ is the
reparameterization by arclength of $\bar{\mathscr{L}}(r_0,s)$ and that
$\frac{d\xi}{ds}=1-r_0\kappa(s)$. Similarly,
\[\frac{ds}{d\xi} = \frac{1}{1-r_0\kappa(s)}\]
which, by our choice of $r_0$, is equal to $\sum_{n=0}^\infty
(r_0\kappa(s))^n$. Next, we compute
\begin{equation*}
\begin{split}
c'(\xi)  = \frac{d}{d \xi} h(\bar{c}(\xi)) &= \frac{\partial}{\partial s} ( h(\gamma_1(s)+r\gamma_2'(s),\gamma_2(s)-r\gamma_1'(s))) \frac{ds}{d\xi}  \\
&= \frac{\partial}{\partial s} \left ( h_0(s) - \frac{r_0}{2}\gamma(s) \cdot \gamma'(s) \right ) \frac{1}{1-r_0\kappa(s)} \\
&= \left ( h_0'(s) -\frac{r_0}{2} -\frac{r_0}{2} \gamma(s)\cdot \gamma''(s) \right ) \frac{1}{1-r_0\kappa(s)}\\
&= \left ( h_0'(s) -\frac{r_0}{2} -\frac{r_0}{2} \gamma(s)\cdot \gamma''(s) \right ) \left ( \sum_{n=0}^\infty (r_0\kappa(s))^n \right )\\
&= h_0'(s) +r_0\alpha(s)+r_0^2\kappa(s)\alpha(s)+ r_0^3\kappa(s)^2\alpha(s) + \dots
\end{split}
\end{equation*}
where $\alpha(s)=-\frac{1}{2} -\frac{1}{2}\gamma(s)\cdot \gamma''(s) + \kappa(s) h_0'(s)$.  At this point we can make some simplifications.  First, we note that as $\kappa(s) = \gamma'' \cdot (\gamma')^\perp$, and  $\gamma'\cdot \gamma'' =0$ (as $|\gamma'(s)|=1$), we have

\[
\gamma''(s)= \kappa(s) (\gamma'(s))^\perp
\]
So, letting $\beta(s)= - \frac{1}{2}\gamma \cdot (\gamma'(s))^\perp + h_0'(s)$,we rewrite $\alpha(s) = -\frac{1}{2} +\kappa(s) \beta(s)$.  Moreover,
\begin{equation*}
\begin{split}
r_0\alpha(s)+r_0^2\kappa(s)\alpha(s)&+ r_0^3\kappa(s)^2\alpha(s) + \cdots
\ =\
r_0\alpha(s) \left ( \sum_{n=0}^\infty (r_0 \kappa(s))^n \right ) \\
&\ =\ \frac{r_0\alpha(s)}{1-r_0\kappa(s)} \\
& \ =\ -\,\left(\frac{r_0}{2} \frac{1}{1-r_0\kappa(s)} - \beta(s) \frac{r_0\kappa(s)}{1-r_0\kappa(s)}\right) \\
&\ =\ -\,\left(\frac{r_0}{2} \frac{1}{1-r_0\kappa(s)} + \beta(s) - \frac{\beta}{1-r_0\kappa(s)}\right) \\
&\ =\
-\,\left(\beta(s) + \frac{r_0 - 2\,\beta(s)}{1 - r_0\,\kappa(s)}\right)\ .
\end{split}
\end{equation*}

We conclude that
\[
c'(\xi)\ =\ h_0'(s)  - \beta(s) - \frac{1}{2} \frac{r_0-2\beta(s)}{1-r_0\kappa(s)}
\]
Since $c'(\xi)$ is again differentiable in $\xi$ and $\xi$ is differentiable in $s$, we conclude, by the chain rule, that $c'(\xi)$ is also differentiable in $s$.  Noting that $h_0'(s)$ and $\beta(s)$ are once differentiable in $s$, we conclude that $(1-r_0\kappa(s))^{-1}$, and hence $\kappa(s)$, is differentiable in $s$.  But, since $\gamma''(s) = \kappa(s) (\gamma'(s))^\perp$, $\gamma''(s)$ is differentiable and hence $\gamma \in C^3(s)$.

Lastly, we examine the impact of the assumption that $S$ contains no
characteristic points on the neighborhood $N$.  Using the parametrization derived above, we
see that the tangent space is spanned by $\nu_H^\perp$ and
\[
\hat{W}=(\gamma_1'(s)+r\gamma_2''(s))\;X_1+(\gamma_2'(s)-r\gamma_1''(s))\;X_2
+
(W_0(s)-r+\frac{r^2}{2}\kappa(s))\;T\]
where, as in the statement of the Theorem, we let
$W_0(s)=h_0'(s)+\frac{1}{2}\gamma'\cdot\gamma^\perp$ and
$\kappa(s)=\gamma''\cdot(\gamma')^\perp$.  $S$ will have a
characteristic point when $<\hat{W},T>=0$,
i.e. when $r=\frac{1\pm\sqrt{1-2W_0(s)\kappa(s)}}{2W_0(s)}$.  Thus, $S$ is
  noncharacteristic if and only if $1-2W_0(s)\kappa(s)<0$.

\end{proof}

Note that, without loss of generality (by simply reparametrizing
$\gamma$), we may assume that any fixed $s \in J$ may be treated as
$s=0$. We will use such a normalization and assume that $J$ is a
neighborhood of $0$.

We wish to examine the behavior of this patch with respect to the notion
of an $X_1$ graph.  Consider the following definitions.

\begin{Def}\label{D:pm} Let $C_1(x_0,y_0,t_0)$ denote the integral curve
of the vector field $X_1$ passing through the point $(x_0,y_0,t_0)$.  In other words,
\[
C_1(x_0,y_0,t_0)=\left \{\left(x_0+r,y_0,t_0-\frac{y_0}{2}\,r\right)\,\Bigl|\, r \in \R\right \}\ .
\]
\end{Def}

Using Definition \ref{D:pm} we next introduce the notion of intrinsic projection of a point to the plane $x=0$.

\begin{Def}\label{D:ipm}  We define the \emph{intrinsic projection map}
\[\Pi(x_0,y_0,t_0)=\{(0,y,t)\} \cap C_1(x_0,y_0,t_0)=(0,y_0,t_0+y_0x_0/2)\ .
\]
\end{Def}

The following equation follows directly from the definition.

\begin{equation}\label{parabola}
\Pi \circ \mathscr{L}(r,s)\ =\ (0,\gamma_2(s)-r\gamma_1'(s),h_0(s)+\frac{1}{2}\gamma_1(s)\gamma_2(s)-r\gamma_1(s)\gamma_1'(s)-\frac{r^2}{2}\gamma_1'(s)\gamma_2'(s))
\end{equation}

\begin{Lem}\label{lemma1}  Let $S$ be a portion of an $H$-minimal
  surface parameterized by a seed curve/height function pair
  $(\gamma(s),h_0(s))$ via \eqref{seedrep} with $r \in \R$, $s\in I$.
  Let $P(s,r)=\Pi \circ \mathscr{L}(r,s)$ be given as in \eqref{parabola}.  There exists
  an interval $J\subset I$ containing so that $P:\R \times J \subset\R^2_{(r,s)} \ra
  \R^2_{(y,t)}$ is a one-to-one $C^2$ diffeomorphism onto its image.
\end{Lem}

\begin{proof}[\textbf{Proof}]
The following properties of the seed curve $\gamma: I \to \R^2$ are essential to our proof.  We gather them here for the sake of convenience.
\begin{itemize}
\item[(i)] $|\gamma'(s)| = 1$.
\item[(ii)] $1 - 2W_0(s)\kappa(s) < 0$.
\item[(iii)]There exists an interval $J\subset I$ such that for all $s\in J$, $\gamma_1'(s) \neq 0$.
\end{itemize}
Properties (i), (ii) and the definitions of $W_0$ and $\kappa$ were establish in Theorem \ref{T:thma}.
Suppose (iii) is not true, then together with (i) we would have $\gamma'(s) = (0,1)1$ for all $s\in I$.  This would implies $\kappa(s) = \gamma''(s)\cdot \gamma'(s)^\perp$ vanishes identically on $I$ and hence (ii) would not be possible.  Therefore, by the continuity of $\gamma_1'$, we can extract a sub-interval $J$ of $I$ on which $\gamma_1'(s) \neq 0$.   To continue we define two auxilary functions $\zeta$ and $\Psi$ by means of $\gamma$ as follows.

\begin{align*}
\zeta:\R\times J \to \R^2\ , \qquad \zeta(r,s) &\ =\ (\gamma_2(s) - r\,\gamma_1'(s), s)\ , \\
\Psi: \zeta(\R\times J)\to \R^2\ ,\qquad (u,v) = \Psi(u,s) &\ =\ \left(u,\sigma(s) + F(s)\,u + \frac{G(s)}{2}\,u^2\right)\ .
\end{align*}
where $F,G,\sigma : J \to \R$ is given by

\begin{align}\label{FGsigma}
F(s) &\ =\ \gamma_1(s) + \frac{\gamma_2(s)\gamma_2'(s)}{\gamma_1'(s)} = \frac{\gamma \cdot \gamma'}{\gamma_1'}\\
\notag
G(s) &\ =\ -\,\frac{\gamma_2'(s)}{\gamma_1'(s)} \\
\notag
\sigma(s) &\ =\ h_0(s) - \frac{1}{2} \gamma_2(s)\,F(s)\ .
\notag
\end{align}
Due to property (iii) above and the the fact that $\gamma \in C^3(I)$, the functions $\zeta, \Psi, F, G, \sigma$ are well defined and are $C^2(J)$.  One can verify by a straight forward computation that
\[
\Pi \circ \mathscr{L}(r,s) \ =\ \Psi \circ \zeta(r,s)\ .
\]
Therefore, if we show that $\Psi \circ \zeta: \R \times J \to \R^2$ is one one then $\Pi\circ \mathscr{L}$ is also one one.  To this end, we will show separately that both $\zeta$ and $\Psi$ are one to one.  The fact that $\zeta$ is one one is easy to verify and follows from the fact that $\gamma_1'(s) \neq 0$ on $J$.  We also note that
\[
\zeta(\R\times J)\ =\ \R\times J\ .
\]
To show that $\Psi$ is one to one, we first consider its second component:
$v(u,s) = \sigma(s) + F(s)u + \frac{G(s)}{2}u^2$.  We have
\[
\frac{\partial}{\partial s} v(u,s)\ =\
\sigma'(s) + F'(s)\,u + \frac{G'(s)}{2}\,u^2\ .
\]
Although it is tedious, nevertheless one can verify by straight forward computations that the following identity holds for any $s\in J$ and any $u\in \R$:

\[
F'(s)^2 - 2\sigma'(s)G'(s)\ =\ 1 - 2W_0(s)\kappa(s) + (|\gamma'(s)|^2 + 1)(|\gamma'(s)|^2 - 1)\ <\ 0\ .
\]
The strict inequality above is due to properties (i) and (ii) of $\gamma$. This in turn implies that the quadratic expression in $u$
\[
\frac{\partial}{\partial s} v(u,s)\ =\
\sigma'(s) + F'(s)\,u + \frac{G'(s)}{2}\,u^2
\]
do not vanish for any fixed $u\in\R$ and any $s\in J$.  Hence we have
\[
\left|\frac{\partial}{\partial s} v(u,s)\right| > 0\ ,\ s\in J
\]
that is, $v(u,s)$ is monotone in $s$ for any fixed $u\in\R$.
We infer from this fact and the definition of $\Psi$ that $\Psi$ is one one.  This completes the proof.
\end{proof}

Several important facts about the functions $F,G,\sigma,\Psi$ were established in the proof of Lemma \ref{lemma1} we single them out here for references.

\begin{Pro}\label{P:Plemma1}
The functions $F,G,\sigma$ satisfy

\begin{equation}\label{E:nonchar}
F'(s)^2 - 2\sigma'(s)G'(s)\ <\ 0\ .
\end{equation}
The function $\Psi:\R\times J\to \R^2$ is invertible on its image.
We let $(u,s) = \Psi^{-1}(u,v)$.
In particular, $s = s(u,v)$ is the second component of $\Psi^{-1}$.
\end{Pro}

These two lemmas show that every $C^2$ noncharacteristic complete
noncompact embedded $H$-minimal surface which is not itself a
vertical plane contains a subsurface which can be written as an
intrinsic graph.  To make the presentation as clean as possible, we
prove an intermediate lemma.

\begin{Lem}\label{lemma3}
Let $S$ be a  $C^2$ noncharacteristic complete noncompact embedded
$H$-minimal surface which is not itself a vertical plane and let
$J$ and the functions
$F,G,\sigma,\Psi$ be the ones from the proof of Lemma \ref{lemma1}
and $s$ as in Proposition \ref{P:Plemma1}.
If $\phi:\Psi(\R\times J)\to \R^2$ is given by
\[
\phi(u,v)\ =\ F(s(u,v))+u G(s(u,v))\quad \text{for } (u,v) \in \Omega = \Psi(\R\times J)\ .
\]
Then
\[
S_0\ =\ \{(0,u,v)\circ(\phi(u,v),0,0)\, |\, (u,v) \in \Omega\}
\]
is a sub surface of $S$.
\end{Lem}

\begin{proof}[\textbf{Proof}]
With the functions $\Psi, \phi, s, F, G, \sigma$ and $\Omega$ as in the
statement of the Lemma, we define $\Phi:\Omega \to \mathbb{H}^1$ as follows
\[
\Phi(u,v) \ =\ \left(\phi(u,v), u, v - \frac{1}{2}\,u\,\phi(u,v)\right)\ .
\]
Our intention is to show that $\Phi(\Omega) = \mathscr{L}(\R\times J)$. We begin by comparing the second components of $\Phi$ and $\mathscr{L}$.  Note that if

\begin{equation}\label{E:u_id}
u\ = \ \gamma_2(s) - r\,\gamma_1'(s)\ ,
\end{equation}
then
\begin{align}\label{phi_id}
\phi(u,v) &\ =\ F(s(u,v)) + u\,G(s(u,v) \\
\notag
&\ =\
F(s) + (\gamma_2(s) - r\,\gamma_1'(s))\,G(s) \\
\notag
\text{(by \eqref{FGsigma})}
&\ =\
\gamma_1(s) + \frac{\gamma_2(s)\gamma_2'(s)}{\gamma_1'(s)} - \Bigl(\gamma_2(s) - r\,\gamma_1'(s)\Bigr)\frac{\gamma_2'(s)}{\gamma_1'(s)} \\
\notag
&\ =\
\frac{\gamma_1(s)\gamma_1'(s) + \gamma_2(s)\gamma_2'(s) - \gamma_2(s)\gamma_2'(s) + r\,\gamma_1'(s)\gamma_2'(s)}{\gamma_1'(s)} \\
\notag
&\ =\ \gamma_1(s) + r\,\gamma_2'(s)\ ,
\notag
\end{align}
which is the first component of $\mathscr{L}$.  We now turn to the third component of $\Phi$.  Keeping in mind that for $(u,v)\in\Omega = \Psi(\R\times J)$ we
have
\[
v\ =\ \sigma(s) + F(s) u + \frac{G(s)}{2}\,u^2
\]
hence

\begin{align*}
v - \frac{1}{2}\,u\,\phi(u,v) & =
\sigma(s) + F(s) u + \frac{G(s)}{2}\,u^2 - \frac{1}{2}\,u\,\phi(u,v) \\
\text{(by \eqref{E:u_id}, \eqref{FGsigma} and \eqref{phi_id})} & =
h_0(s) - \frac{1}{2} \gamma_2(s)\left(\gamma_1(s) +
\frac{\gamma_2(s)\gamma_2'(s)}{\gamma_1'(s)}\right) \\
& + \left(\gamma_1(s) +
\frac{\gamma_2(s)\gamma_2'(s)}{\gamma_1'(s)}\right)(\gamma_2(s) -
r\,\gamma_1'(s)) \\ &  -
\frac{1}{2}\frac{\gamma_2'(s)}{\gamma_1'(s)}(\gamma_2(s) -
r\,\gamma_1'(s))^2 - \frac{1}{2}(\gamma_2(s) -
r\,\gamma_1'(s))(\gamma_1(s) + r\,\gamma_2'(s)) \\ & = h_0(s) -
\frac{r}{2} \gamma(s)\cdot\gamma'(s) \notag
\end{align*}
which is the third component of $\mathscr{L}$.
\end{proof}

Finally, we turn to the

\begin{proof}[\textbf{Proof of Theorem \ref{T:existstrip0}}]
Since $S$ is not itself a vertical plane, Lemma \ref{L:lemma2} guarantee the
existence of a point $g_o\in S$ and a neighborhood $N$ of $g_o$ such that $N$ can be written as a graph over the plane $t = 0$.  Theorem \ref{T:thma} then provides the necessary parameterization of such a neighborhood by the map $\mathscr{L}$ whose domain is $\R\times J$.  Lemmas \ref{lemma1}, \ref{lemma3} and Proposition \ref{P:Plemma1} then show that the portion $\mathscr{L}(\R\times J)\subset S$ can
be reparameterized to conform to Definition \ref{intdeltastrip} hence, establishing the required $\delta$-graphical strip.
\end{proof}

Combining this with Theorem \ref{T:existstrip0}, we can now easily
prove the main Theorem.

\vspace{.2in}

\noindent

\begin{proof}[\textbf{Proof of Theorem \ref{I:main}}]
Suppose $S$ is a $C^2$ complete embedded noncharacteristic
$H$-minimal surface without boundary which is not a vertical plane.
Then, Theorem \ref{T:existstrip0} shows that $S$ contains an
intrinsic graphical strip, $S_0$, and thus, by Theorem
\ref{I:unstable}, $S_0$, and hence $S$, is not stable.

\end{proof}

\end{document}